\documentclass[a4paper,UKenglish]{article}

\usepackage[utf8]{inputenc}
\usepackage{amssymb,amsmath,amsthm}
\usepackage{xifthen}
\usepackage{xparse}
\usepackage{arydshln}
\usepackage{ifmtarg}
\usepackage{graphicx}
\usepackage{tikz}

\usepackage{thmtools}
\usepackage{thm-restate}

\usepackage{todonotes}
\usepackage{xcolor}
\usepackage{latexsym}

\title{Complementary cycles of any length in regular bipartite tournaments}
\author{St\'ephane Bessy\thanks{This work was supported by ANR under contract DIGRAPH ANR-19-CE48-0013-02} \\ LIRMM, Univ Montpellier, CNRS, France\\
  {\tt stephane.bessy@lirmm.fr}\\[2ex]
Jocelyn Thiebaut\\
Univ. Orléans, INSA Centre Val de Loire,\\ LIFO EA 4022, F-45067 Orléans, France\\
{\tt jocelyn.thiebaut@univ-orleans.fr}\\}

\theoremstyle{plain}
\newtheorem{theorem}{Theorem}
\newtheorem{lemma}{Lemma}
\newtheorem{corollary}[lemma]{Corollary}

\newtheorem{claim}{Claim}[lemma]

\theoremstyle{definition}

\newtheorem{conjecture}[lemma]{Conjecture}

\newcommand{\rb}{\}}
\newcommand{\lb}{\{}

\begin{document}

\maketitle
\begin{abstract}
 Let \(D\) be a \(k\)-regular bipartite tournament on $n$ vertices. We show that, for
 every \(p\) with \mbox{\(2\le p\le n/2-2\)}, \(D\) has a cycle
 \(C\) of length \(2p\) such that \(D\setminus C\) is hamiltonian
 unless \(D\) is isomorphic to the special digraph \(F_{4k}\). This
 statement was conjectured by Manoussakis, Song and Zhang [K.  Zhang,
   Y.  Manoussakis, and Z.  Song. Complementary cycles containing a
   fixed arc in diregular bipartite tournaments.  {\it Discrete
     Mathematics}, 133(1-3):325--328,1994]. In the same paper, the
 conjecture was proved for $p=2$ and more recently Bai, Li and He gave
 a proof for $p=3$ [Y. Bai, H. Li, and W. He.  Complementary cycles in
   regular bipartite tournaments. {\it Discrete Mathematics},
   333:14--27, 2014].
\end{abstract}


\section{Introduction}

Throughout all the paper, we are dealing with directed graphs or digraphs. Notations
not explicitly stated follows~\cite{BJG09}.

A \textit{cycle-factor} of a digraph \(D\) is a spanning subdigraph of
\(D\) whose components are vertex-disjoint (directed) cycles. For some
positive integer \(k\), a \textit{\(k\)-cycle-factor} of \(D\) is a
cycle-factor of \(D\) with \(k\) vertex-disjoint cycles; it can also
be considered as a partition of \(D\) into \(k\) hamiltonian
subdigraphs. In particular, a 1-cycle-factor is a hamiltonian cycle of
\(D\). The cycles of a 2-cycle-factor are often called
\emph{complementary cycles}.

A \textit{tournament} is an orientation of a complete graph. A lot of
work has been done on cycle-factors in tournament. For
instance, the classical result of Camion~\cite{C59} states that a
tournament is strong if and only if it admits an hamiltonian cycle
(i.e.\ a 1-cycle-factor). Reid~\cite{R85} proved that every
2-connected tournament with at least 6 vertices and not isomorphic to
\(T_7\) has a 2-cycle-factor, where \(T_7\) is the Paley tournament on
7 vertices: it has vertex set $\{1,2,3,4,5,6,7\}$ and is the union of
the three directed cycles: the cycle $1,2,3,4,5,6,7$, the cycle $1,3,5,7,2,4,6$ and the cycle $1,5,2,6,3,7,4$. This
  was then extended by Chen, Gould and Li~\cite{C01} who proved that every
  \(k\)-connected tournament with at least \(8k\) vertices contains a
  \(k\)-cycle-factor.

On the other hand, finding cycles of many lengths in different digraphs
is a natural problem in Graph Theory~\cite{BT81}. For example, Moon
proved in~\cite{M66} that every vertex of a strong tournament is in a
cycle of every length. Concerning cycle-factor with prescribed lengths
in tournaments, Song~\cite{S93}, extended the results of
Reid~\cite{R85}, and proved that every 2-connected tournament with at
least 6 vertices and not isomorphic to \(T_7\) has a 2-cycle-factor
containing cycles of lengths \(p\) and \(|V(T)|-p\) for all \(p\) such
that \(3 \leq p \leq |V(T)|-3\). Li and Shu~\cite{LS05} finally
refined the previous result by proving that any strong tournament with
at least 6 vertices, a minimum out-degree or a minimum in-degree at
least 3, and not isomorphic to \(T_7\), has 2-cycle-factor containing
cycles of lengths \(p\) and \(|V(T)|-p\) for all \(p\) such that \(3
\leq p \leq |V(T)|-3\). Recently, K\"uhn, Osthus and
Townsend~\cite{KO16} have extended these results by showing that every
\(O(k^5)\)-connected tournament admits a \(k\)-cycle-factor with
prescribed lengths.

\medskip

In this paper, we focus on cycle-factors in \(k\)-regular bipartite
tournaments. A {\it \(k\)-regular bipartite tournament} is an
orientation of a complete bipartite graph $K_{2k,2k}$ where every
vertex has out-degree $k$ exactly. The existing results concerning
this class of digraphs try to extend what is known about cycle-factors
in tournaments. Thus, Zhang and Song~\cite{SZ88} proved that
any \(k\)-regular bipartite tournament with \(k \geq 2\) has a
2-cycle-factor. Moreover, Manoussakis, Song and Zhang~\cite{ZMS94}
conjectured the following statement whose proof is the main result of
this paper.

\begin{theorem}%
 \label{theo:2cf}
 For $k\ge 2$ let \(D\) be a \(k\)-regular bipartite tournament not
 isomorphic to \(F_{4k}\). Then for every \(p\) with \(2\le p\le k\),
 \(D\) has a 2-cycle-factor containing cycles of length \(2p\) and
 \(|V(D)|-2p\).
\end{theorem}

The digraph \(F_{4k}\) corresponds to the \(k\)-regular bipartite
tournament consisting of four independent sets \(K, L\), \(M\) and
\(N\) each of cardinality \(k\) with all possible arcs from \(K\) to
\(L\), from \(L\) to \(M\), from \(M\) to \(N\) and from \(N\) to
\(K\). In fact, every cycle of \(F_{4k}\) has length \(0 \pmod 4\). Thus,
for instance, \(F_{4k}\) has no 2-cycle-factor of length 6 and
\(4k-6\). Zhang, Manoussakis and Song proved their conjecture when \(p
= 2\) in their original paper~\cite{ZMS94}. In 2014, Bai, Li and He
proved the conjecture for \(p =3\)~\cite{BLH14}.
\bigskip

Our proof of Theorem~\ref{theo:2cf} runs by induction on $p$ and so we
will use as basis cases the results of Theorem~\ref{theo:2cf} for
\(p=2\)~\cite{ZMS94} and \(p=3\)~\cite{BLH14}.  To perform induction
step, we will need also a weaker form of Theorem~\ref{theo:2cf} given
by the following lemma.

\begin{restatable}{lemma}{cftotcf}
 \label{lem:cf-to-2cf}
 For $k\ge 2$ let \(D\) be a \(k\)-regular bipartite tournament. If
 \(D\) contains a cycle-factor with a cycle \(C\) of length \(2p\)
 with \(2\leq p\leq k\), then \(D\) contains a \((2p, |V(D)|-
 2p)\)-cycle-factor \((C', C'')\). Moreover, if \(p\) is at least 3
 and even and \(D[C]\) is not isomorphic to \(F_{2p}\), then \(D[C']\)
 is not isomorphic to \(F_{2p}\) neither
\end{restatable}

Theorem~\ref{theo:2cf} and Lemma~\ref{lem:cf-to-2cf} will both need
the following result due to H\"aggkvist and Manoussakis to be proven.
\begin{theorem}[H\"aggkvist and Manoussakis~\cite{HM89} and
  Manoussakis~\cite{M87}]%
 \label{theo:haggvist-manoussakis}
 A bipartite tournament containing a cycle-factor has either a
 hamiltonian cycle or a cycle-factor consisting of cycles \(C_1, \dots
 , C_m\) such that for any \(1\le i<j\le m\), there is no arc from
 \(C_j\) to \(C_i\).
\end{theorem}

Section~\ref{sec:def} contains introducing tools and definitions we
use for the proofs. Lemma~\ref{lem:cf-to-2cf} and
Theorem~\ref{theo:2cf} are proven in Section~\ref{sec:lem}
and~\ref{sec:thm} respectively.  Finally, in Section~\ref{sec:cr} we
give some concluding remarks concerning cycle-factors in bipartite
tournaments.

\section{Definitions and Notations}%
\label{sec:def}
\paragraph*{Generic definitions} 
Throughout the paper, all digraphs are simple and loopless. Notations
not given here are consistent with~\cite{BM08}. The vertex set of a
digraph \(D\) is denoted by \(V(D)\) and its arcs set by
\(A(D)\). Given a digraph \(D\) and a set \(X\) of vertices such that
\(X \subseteq V(D)\), we denote by \(D[X]\) the subdigraph with vertex
set \(X\), and arc set \(\lb{}uv \in A(D) : u \in X, v\in X\rb{}\). If
\(H\) is a subdigraph of \(D\) we abusively write \(D[H]\) for
\(D[V(H)]\). In the following, we say that two digraphs \(D_1\) and
\(D_2\) are isomorphic if there exists a bijection \(\varphi : V(D_1)
\rightarrow V(D_2)\) such that, for every ordered pair \(x, y\) of
vertices in \(D_1\), \(xy\) is an arc of \(D_1\) if and only if
\(\varphi(x)\varphi(y)\) is an arc of \(D_2\).

The \emph{complement digraph} of a digraph \(D\), denoted by
\(\overline{D}\), corresponds to the digraph with vertex set \(V(D)\)
and arc set \(\lb{}uv : uv \notin A(D)\rb{}\).
\medskip

For any vertices \(u\) and \(v\) such that \(uv\) is an arc of \(D\),
we say that \(v\) is an \emph{out-neighbor} of \(u\) and \(u\) is an
\emph{in-neighbor} of \(v\). The \emph{out-neighborhood}
(resp. \emph{in-neighborhood}) of \(u\) in \(D\), denoted
\(N^+_D(u)\) (resp. \(N^-_D(u)\)), corresponds to the set of vertices
which are out-neighbor (resp. \emph{in-neighbor}) of \(u\). The
\emph{out-degree} (resp. \emph{in-degree}) of a vertex \(u\), denoted
\(d^+_D(u)\) (resp. \(d^-_D(u)\)), is the size of its
\emph{out-neighborhood} (resp. \emph{in-neighborhood}). We say that \(D\)
is \emph{regular} if, for any vertices $u$ and $v$ of $D$, we have $d^+(u) = d^-(u) = d^+(v) = d^-(v)$. If, in addition, we have
\(d^+(u)=d^-(u)=k\), we say that $D$ is \emph{$k$-regular}. For two vertex-disjoint sets \(A\) and \(B\), if
there are all the possible arcs going from \(A\) to \(B\), then we say
that \(A\) \emph{dominates} \(B\). The number of arcs from \(A\) to
\(B\) is denoted by \(e(A,B)\). We simply write \(e(A)\) instead of
\(e(A,A)\) to denote the number of arcs linking two vertices of \(A\).

Similarly, if there is no arc from \(u\) to \(v\), we say that there
is an \emph{anti-arc} from \(u\) to \(v\). Moreover, we say that \(v\)
is an \emph{anti-out-neighbor} of \(u\) and \(u\) is an
\emph{anti-in-neighbor} of \(v\). The \emph{anti-out-neighborhood}
(resp. \emph{anti-in-neighborhood}) of \(u\) in \(D\), denoted
\(\overline{N^+_D}(u)\) (resp. \(\overline{N^-_D}(u)\)), corresponds
to the set of vertices which are anti-out-neighbor
(resp. \emph{anti-in-neighbor}) of \(u\). The \emph{anti-out-degree}
(resp. \emph{anti-in-degree}) of a vertex \(u\), denoted
\(\overline{d^+_D}(u)\) (resp. \(\overline{d^-_D}(u)\)), is the size
of its anti-out-neighborhood
(resp. \emph{anti-in-neighborhood}). For two vertex-disjoint sets
\(A\) and \(B\) if there are all the possible anti-arcs going from \(A\) to
\(B\) (that is there are no arcs from \(A\) to \(B\)), we say that
\(A\) \emph{anti-dominates} \(B\).
 
If there is no ambiguity, we omit the reference to the considered
digraph in the previous notations (\(N^+(u)\) instead of \(N^+_D(u)\),
\emph{etc}...).\medskip

Given a digraph \(D\) and a set \(\lb{}u_1, \dots, u_t\rb{}\) of \(t\)
disjoint vertices of \(D\), we say that \(P=u_1, \dots, u_t\) is a
\emph{directed path} of \emph{length} $t-1$ of \(D\) if \(u_i u_{i+1}
\in A(D)\) for \(1 \leq i \leq t-1\). The vertices $u_2,\dots
,u_{t-1}$ are called the \emph{internal vertices} of $P$. In addition,
if we also have \(u_t u_1 \in A(D)\), then \(u_1, \dots, u_t\) is a
\emph{directed cycle} of \emph{length} $t$. A cycle of length 2 is
also called a {\it digon}. In the paper, \emph{path} and \emph{cycle}
always means directed path and directed cycle,
respectively. Symmetrically, given a digraph \(D\) and a set of \(t\)
disjoint vertices \(\lb{}u_1, \dots, u_t\rb{}\) of $D$, we say that
\(u_1,\dots, u_t\) is an \emph{anti-path} if \(u_{i}u_{i+1} \notin
A(D)\) for any \(1 \leq i \leq t-1\). In addition, if \(u_t u_1 \notin
A(D)\), then we obtain an \emph{anti-cycle}. A digraph \(D\) is
\emph{strongly connected} (or \emph{strong} for short) if we have a
path from $u$ to $v$ for any vertices \(u\) and \(v\) of \(D\). If $D$
is not strong, a \emph{strongly connected component} (or \emph{strong
  component} for short) of $D$ is a set $X$ of vertices of $D$ such
that $D[X]$ is strong and $X$ is maximal by inclusion for that. A
strong component $X$ is an \emph{initial strong component} (resp. a
\emph{terminal strong component}) of $D$ if there is no arc from
$V\setminus X$ to $X$ (resp. from $X$ to $V\setminus X$) in $D$. It is
well-known that every non-strong digraph contains at least one initial
and one terminal strong component.  Given a set \(X\) of vertices and
a cycle \(C\), we denote by \(C(X)\) the set of the successors of
\(X\) along \(C\). If \(X\) is a singleton \(\lb{}x\rb{}\), we simply
write \(C(x)\) instead of \(C(\lb{}x\rb{})\).

Finally for a digraph \(D\) and integers \(n_1, \dots, n_k\) such that
\(n_1+\cdots + n_k= |V(D)|\), a \((n_1,\dots, n_k)\)\emph{-cycle
 -factor} is a \(k\)-cycle-factor $(C_1,\dots ,C_k)$ of $D$ such that
for each $i=1,\dots ,k$ the cycle $C_i$ has length $n_i$. The cycle
$C_1$ will be called the \emph{first} cycle of the cycle-factor.

\paragraph*{Bipartite tournaments and contracted digraphs} A \emph{bipartite tournament} is an orientation of a complete
bipartite graph. Let \(D\) be a \(k\)-regular bipartite tournament
with bipartition \((S,T)\). We have \(|S|=|T|=2k\), and for any vertex
\(u\) of \(D\) we have \(d^+(u)=d^-(u)=k\). Moreover, the (unoriented)
graph on \(S\cup T\) containing an edge for every arcs from \(S\) to
\(T\) is a bipartite graph where every vertex has degree \(k\). Hence,
by Hall's Theorem~\cite{BM08}, it admits a perfect matching. Let \(M\)
be a set of arcs of \(D\) corresponding to such a perfect
matching. For each vertex \(u\) of \(S\), the vertex \(M(u)\) denotes
the only vertex of \(T\) such that the arc \(uM(u)\) is an arc of
\(M\).

We extend this notation to sets that is, given a subset \(X\) of
\(S\), we define \(M(X)\) by \(M(X) = \bigcup_{x\in X} M(x)\).  \par
\medskip
Now, given a perfect matching \(M\) of $D$ made of arcs from $S$ to
$T$, we define the \emph{contracted digraph according to} \(M\),
denoted \(D^M\) and obtained by contracting the arcs of \(M\) and only
keeping the arcs of \(D\) from \(T\) to \(S\). More formally, the new
digraph \(D^M\) has vertex set \(S\) and arc set \mbox{\(\lb{}uv
  \colon u\in S, v\in S \textrm{ and } M(u)v \in A(D)\rb{}\)}. As the
vertex set of \(D^M\) is \(S\), we also consider vertices of \(D^M\)
as vertices of \(D\). Notice that \(D^M \) has \(2k\) vertices and
that for every vertex \(u\) of \(D^M\) we have \(N^+_{D^M} (u) =
N^+_{D} (M (u))\) and so, $u$ has out-degree \(k\) exactly. Similarly,
\(u\) has in-neighborhood \(\lb{} v\in S : M(v)\in N^{-}_D(u)\rb{}\)
and so has in-degree \(k\) exactly. Notice also that \(D^M\) does not
contain any parallel arc but may contains cycles on 2 vertices. See Figure~\ref{fig:contracteddigraph2} which depicts an example of contracted digraph according a matching.
\par
\begin{figure}
	\centering
	\includegraphics[width=0.8\linewidth]{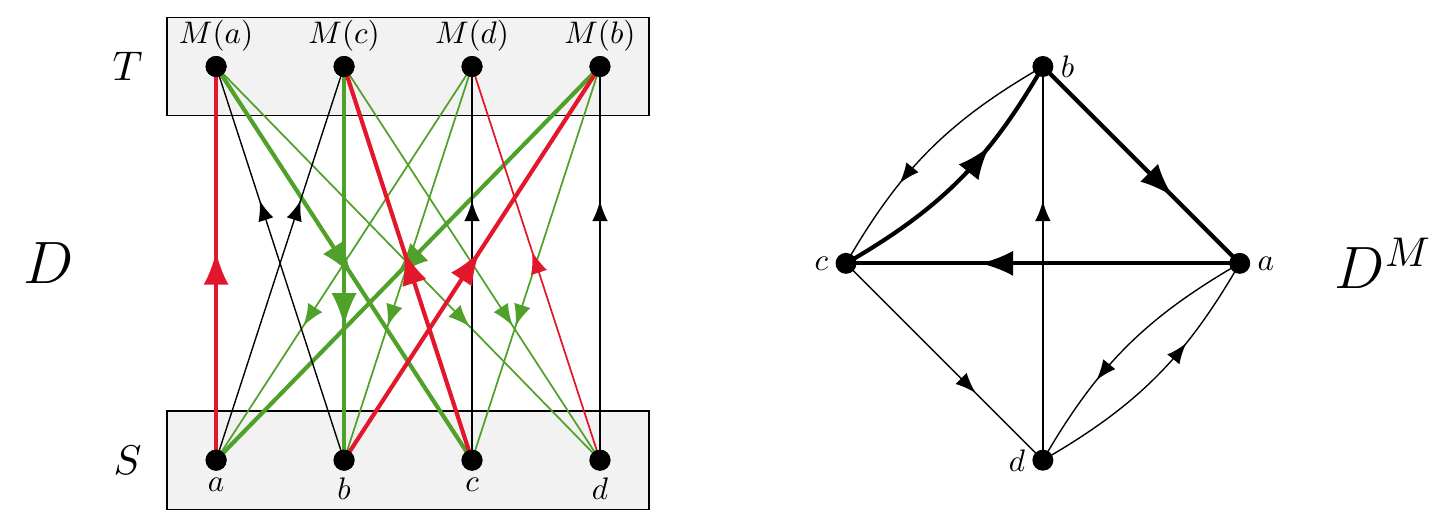}
	\caption{A 2-regular bipartite tournament $D$ and the contracted digraph according to the red matching $M$. In $D^M$, we only keep the vertices of $S$ and the green arcs from $T$ to $S$. Note that the cycle $a,c,b$ in $D^M$ (with bold arcs) corresponds to the cycle $a,M(a),c,M(c),b,M(b)$ in $D$ (also depicted with bold arcs).}
	\label{fig:contracteddigraph2}
\end{figure}

Let
$D'$ be a sub-digraph of $D$ and denote by $M'$ the arcs of $M$ with
both extremities in $D'$. If $M'$ is also a perfect matching of $D'$,
then we abusively denote by $D'^M$ the contracted digraph
$D'^{M'}$.\\ If now $M$ is a perfect matching of $D$ made of arcs from
$T$ to $S$, then we can symmetrically define $D^M$ by exchanging $S$
and $T$ in the previous definitions.  \par

Structurally, \(u_1,\dots ,u_t\) is a cycle in \(D^M\) if and only if
\(u_1,M(u_1),\dots ,u_t,M(u_t)\) is a cycle of \(D\). Thus, to prove
Theorem~\ref{theo:2cf}, if \(D\) is not isomorphic to \(F_{4k}\), then
for every \(p\) with \(2\leq p \leq k-2\), it suffices to find a \((p,
2k-p)\)-cycle-factor in \(D^M\). Finally, note that the
graph \(D^M\) contains the same information than \(D\) but, most of
the time, it will be easier to identify particular structures in the
former.\par

\section{From cycle-factor to 2-cycle-factor}%
\label{sec:lem}

The aim of this section is to prove Lemma~\ref{lem:cf-to-2cf} which
states that, given a cycle-factor with a cycle of length \(2p\), we
can ``merge'' the other cycles in order to obtain a \((2p, |V(D)|-
2p)\)-cycle-factor.  Moreover, in the case where $p$ is even, we could
ask that the new cycle of length \(2p\) is not isomorphic to
\(F_{2p}\) is the former was not. This condition will be useful in the
induction step to prove Theorem~\ref{theo:2cf}.

\cftotcf*

\begin{proof} As the cases where \(p=2\) and \(p=3\) of Theorem~\ref{theo:2cf}
  are already proven in~\cite{ZMS94} and~\cite{BLH14}, we assume that
  \(p\ge 4\).\\ Consider a cycle-factor $\cal C'$ of \(D\) containing
  a cycle $C$ of length \(2p\), such that $D[C]$ is not isomorphic to
  $F_{2p}$ if $p$ is even, and such that $\cal C'$ has a minimum total
  number of cycles. We denote by \(\cal C\) the set of cycles of $\cal
  C'$ different from $C$. Thus, we want to show that \(|{\cal C}|=1\).
  By Theorem~\ref{theo:haggvist-manoussakis}, if \(|{\cal C}| \neq 1\)
  then we can assume that \({\cal C}=\lb{}C_1, \dots , C_{\ell}\rb{}\)
  with \(\ell \ge 2\) and that \(C_i\) dominates \(C_j\) whenever
  \(i<j\). Let \((S,T)\) denotes the bipartition of \(D\) and for
  every \(i\), we denote by \(c_i\) the number of vertices of
  \(V(C_i)\cap S\), that is \(C_i\) is of length \(2c_i\).
 
 \begin{claim}%
  \label{claim:nb-of-arc-c-ci}
  We have \(e(C,C_1)=c_1(2k-c_1)\) and \(e(C_{\ell},C) = c_{\ell}(2k-c_{\ell})\).
 \end{claim}
 
 \begin{proof} A vertex \(x_1\) in \(C_1\) is an in-neighbor of every vertex in  \(C_2, \dots ,C_{\ell}\).
   Hence, we have \(\sum_{x\in C_1}d^-_{D}(x)=e(C_1)+e(C,C_1)\).  Thus
   we get \(2kc_1=c_1^2+e(C,C_1)\) and the first result holds.  The
   other equality is obtained similarly by reasoning on the out-neighborhood of $C_\ell$.
 \end{proof}
 
 \medskip
Now, using Claim~\ref{claim:nb-of-arc-c-ci} we have the following.
 \[\Big(\ \frac{1}{c_{\ell}}\sum_{x\in T\cap C_{\ell}} d^+_{C}(x)+\frac{1}{c_1}\sum_{x\in T\cap C}d^+_{C_1}(x)\ \Big)+\Big(\ \frac{1}{c_{\ell}}\sum_{x\in S\cap C_{\ell}}d^+_{C}(x)+\frac{1}{c_1}\sum_{x\in S\cap C}d^+_{C_1}(x)\ \Big)\]
 \[=\frac{1}{c_{\ell}}e(C_{\ell},C)+\frac{1}{c_1}e(C,C_1)=4k-(c_1+c_{\ell})\]
 Hence, either we have
 \[\Big(\ \frac{1}{c_{\ell}}\sum_{x\in T\cap C_{\ell}} d^+_{C}(x)+\frac{1}{c_1}\sum_{x\in T\cap C}d^+_{C_1}(x)\ \Big)\ge 2k-\frac{(c_1+c_{\ell})}{2}\] or we have
 \[\Big(\ \frac{1}{c_{\ell}}\sum_{x\in S\cap C_{\ell}} d^+_{C}(x)+\frac{1}{c_1}\sum_{x\in S\cap C}d^+_{C_1}(x)\ \Big)\ge 2k-\frac{(c_1+c_{\ell})}{2}\]

 Without loss of generality, we can assume that the former holds
 (otherwise we exchange in that follows the role of \(S\) and \(T\)). \smallskip
 
 Denote by \(M\) the set of arcs of the digraph induced by the cycle-factor \(C, C_1, \ldots, C_\ell{}\) and going from \(S\) to \(T\) in
 \(D\). It is clear that \(M\) forms a perfect matching of \(D\) and
 that \(C^M \cup {\cal C}^M\) is a cycle-factor of \(D^M\), where
 \({\cal C}^M = \lb{} C^M_1, \ldots, C^M_\ell{}\rb{}\).  Moreover,
 notice that the length of \(C^M\) is \(p\) and for \(i\) with \(1\leq
 i \leq \ell \) the length of \(C_i^M\) is \(c_i\).  By the previous
 assumption, in \(D^M\) we have the following
 \begin{equation}
  \begin{split}
   \frac{e(C_{\ell}^M,C^M)}{c_{\ell}}+\frac{e(C^M,C_1^M)}{c_1}\ge 2k-
   \frac{(c_1+c_{\ell})}{2}
   \label{eq:assumption}
  \end{split}
 \end{equation}
 
 Now, we will find suitable vertices in \(C^M\) to design the desired
 2-cycle-factor. To do so, let $W$ (resp. $R$) be the set
 of pairs \(\lb{}x,y\rb{}\) of distinct vertices of \(C^M\) which are
 ``well connected'' to \(C_1^M\) (resp.\ \emph{from} \(C_{\ell}^M\)), that is
 such that \mbox{\(d^+_{C_1^M}(x)+d^+_{C_1^M}(y)>c_1\)} (resp.
 \mbox{\(d^-_{C_{\ell}^M}(x)+d^-_{C_{\ell}^M}(y)>c_{\ell}\)}).  We
 denote by \(w\) (resp. \(r\)) the cardinal of \(W\)
 (resp. \(R\)).

 \begin{claim}%
  \label{claim:bound-on-b}
  We have \(w+r\geq\frac{p(p-1)}{2}\).
 \end{claim}
 
 \begin{proof} For every pair \(\lb{}x,y\rb{}\) of distinct vertices of \( V(C^M)\), we have \(d^+_{C_1^M}(x)+d^+_{C_1^M}(y)\le 2c_1\) and, if \(\lb{}x,y\rb{}\) is not a pair of \(W\), we have more precisely \(d^+_{C_1^M}(x)+d^+_{C_1^M}(y)\le c_1\).  Thus, in total,
  \begin{equation}
   \begin{split}
    \sum_{\substack{\lb{}x,y\rb{} \textrm{ pair }\\ \textrm{of }
        V(C^M)} } \Big(d^+_{C_1^M}(x)+d^+_{C_1^M}(y) \Big) &
    =\sum_{\lb{}x,y\rb{} \in W}\Big(d^+_{C_1^M}(x)+d^+_{C_1^M}(y)\Big)+\sum_{\lb{}x,y\rb{}
      \notin W}\Big(d^+_{C_1^M}(x)+d^+_{C_1^M}(y)\Big)\\ & \le
    2wc_1+\Big(\frac{p(p-1)}{2}-w\Big)c_1
    \label{eq:pairs-of-B1}
   \end{split}
  \end{equation}
  and
  \[\sum_{\substack{\lb{}x,y\rb{} \textrm{ pair }\\ \textrm{of } V(C^M)}} \Big(d^+_{C_1^M}(x)+d^+_{C_1^M}(y)\Big)=(p-1)e(C^M,C_1^M)\]
  Thus, we get \[(p-1)\frac{e(C^M,C_1^M)}{c_1}\le
  w+\frac{p(p-1)}{2}\] Similarly, if we do the same reasoning on the arcs from $C_{\ell}^M$ to $C^M$ and $R$, we
  obtain \[(p-1)\frac{e(C_{\ell}^M,C^M)}{c_{\ell}}\le
  r+\frac{p(p-1)}{2}\] Hence, using the
  inequality~\eqref{eq:assumption} we
  have \[(p-1)(2k-\frac{c_1+c_{\ell}}{2})\le w+r+p(p-1)\]
  Finally, since \(C^M\cup C_1^M\cup C_{\ell}^M\) is a subgraph of
  \(D^M\), we have \(p+c_1+c_{\ell}\le 2k\) and so
  \(2k-(c_1+c_{\ell})/2\ge k+p/2\). With the previous inequality we
  obtain \mbox{\(w+r\ge (p-1)(k-p/2)\)}. Finally, using that
  \(k\ge p\), we get the result, that is \mbox{\(w+r\ge
    p(p-1)/2\)}.
 \end{proof}
 \medskip
 
 Now, for every pair \(\lb{}x,x'\rb{}\) of distinct vertices of
 \(C^M\), we color \(\lb{}x,x'\rb{}\) in white if it is a pair of
 \(W\), and we color \(\lb{}x,x'\rb{}\) in red if
 \(\lb{}y,y'\rb{}\in R \) where \(y\) (resp. \(y'\)) is the
 out-neighbor of \(x\) (resp. \(x'\)) along \(C^M\).

 \begin{claim}%
  \label{claim:colouring}
  There exists a pair of vertices colored both in white and red.
 \end{claim}
 
 \begin{proof} If \(w+r>p(p-1)/2\), then we have colored more than
   \(p(p-1)/2\) pairs of distinct vertices of \(C^M\). Thus, at least
   one pair have been colored both in white and red, yielding the
   result.
  
  Now, let suppose that \(w+r \le p(p-1)/2\). By
  Claim~\ref{claim:bound-on-b}, it means that we have \(w+r
  =p(p-1)/2\) and that all the inequalities leading to the proof of
  Claim~\ref{claim:bound-on-b} are equalities. In particular, we have
  \(p+c_1+c_\ell=2k\) and \(p=k\). Notice that, as \(c_1\) and
  \(c_\ell \) are at least 2, we have \(k\ge 4\). Moreover,
  as~\eqref{eq:pairs-of-B1} is also an equality, we have
  \(d^+_{C_1^M}(x)+d^+_{C_1^M}(y)= 2c_1\) for every pair
  \(\lb{}x,y\rb{}\) of \(W\) and
  \(d^+_{C_1^M}(x)+d^+_{C_1^M}(y)=c_1\) for every pair
  \(\lb{}x,y\rb{}\) of vertices of \(C_1^M\) which is not in
  \(W\). In particular, if \(\lb{}x,y\rb{}\in W\) we have exactly
  \(d^+_{C_1^M}(x)=c_1\) and \(d^+_{C_1^M}(y)=c_1\).  Similarly, if
  \(\lb{}x,y\rb{}\in R\), then we have
  \(d^-_{C_\ell^M}(x)=c_\ell\) and \(d^-_{C_\ell^M}(y)=c_\ell\).  In
  particular, we can prove that $w\neq 0$. Indeed, if it is not the case, we
  have $r=p(p-1)/2$, that is, every pair of elements of $C^M$ is
  a pair of $R$. Then by the previous remark, every vertex $x$ of
  $C^M$ satisfies \(d^-_{C_\ell^M}(x)=c_\ell\), and $C_\ell^M$
  dominates $C^M$. So, the out-neighborhood of any vertex $y$ of
  $C_\ell^M$ would contain the successor of $y$ along $C_\ell^M$ and
  all the cycle $C^M$, which is of length $p=k$, a contradiction to
  $d^+_{D^M}(y)=k$. Similarly, we have $r\neq 0$.\par\smallskip Now, let
  \(V_W\) (\emph{resp.} \(V_R \)) be the collection of vertices
  which belong to at least one pair of \(W\) (\emph{resp.} \(R
  \)).  Thus, $V_W$ is not empty and for every vertex \(v\in V_W\), we
  have \(d^+_{C_1^M}(v) = c_1\).  As every pair $\{x,y\}$ of distinct
  vertices of $C^M$ satisfies $d^+_{C_1^M}(x)+d^+_{C_1^M}(y)\in
  \{c_1,2c_1\}$, it is easy to see that every vertex \(w \notin V_W\)
  satisfies \(d^+_{C_1^M}(w) = 0\), and that there is at most one
  vertex $a$ which does not belong to $V_W$.  With the same
  arguments we see that there is at most one vertex $b$ which does
  not belong to $V_R$. So, as $p=k\ge 4$ there exists a pair of
  vertices of $C^M$ containing neither $a$ nor the predecessor of
  $b$ along $C^M$. Thus, this pair will colored both in white and red.
 \end{proof}
 \medskip
 In the following, let \(\lb{}y_1, z_1\rb{}\) be a pair of vertices of
 \(V(C^M)\) colored both in white and red and we will denote by
 \(y_\ell \) (resp. \(z_\ell \)) the successor of \(y_1\)
 (resp. \(z_1\)) along \(C^M\). Therefore we have \(\lb{}y_1, z_1\rb{}
 \in W\) and \(\lb{}y_\ell, z_\ell\rb{} \in R \). Notice that
 \(\lb{}y_1, z_1\rb{}\) and \(\lb{}y_\ell, z_\ell\rb{}\) are two
 distinct pairs of vertices but that \(y_\ell = z_1\) or \(z_\ell =
 y_1\) is possible.

 \begin{claim}%
  \label{claim:pathlenght}
  For every \(i\) with \(0 \leq i \leq c_1 - 2\) there exist \(y\)
  and \( z \in C_1^M\) with \(y_1y, z_1z \in A(D^M)\) and such that
  the sub-path of \(C_1^M\) from $y$ to $z$ contains $i$ internal
  vertices exactly.  Similarly for every \(i'\) with \(0 \leq i' \leq
  c_\ell - 2\) there exist \(y'\) and \( z' \in C_\ell^M\) with
  \(y'y_\ell, z'z_\ell \in A(D^M)\) and such that the sub-path of
  \(C_\ell^M\) from $z'$ to $y'$ contains $i'$ internal vertices
  exactly.
  \end{claim}
 
 \begin{proof} Suppose that for every vertex \(y\) of \(N^+_{C_1^M}(y_1)\),
   the vertex \(z\) which is \(i+1\) vertices away from \(y\) along
   \(C_1^M\) does not belong to \(N^+_{C_1^M}(z_1)\). Thus we have
   \(|N^+_{C_1^M}(z_1)|\leq c_1 - |N^+_{C_1^M}(y_1)|\), which
   contradicts \(d^+_{C_1^M}(y_1)+d^+_{C_1^M}(z_1)>c_1\) as
   \(\lb{}y_1, z_1\rb{}\) is a pair of \(W\). The proof is similar
   for the pair $\{y_\ell,z_\ell\}$.
 \end{proof}

Now, we can construct our 2-cycle-factor from the collection of
cycles.  To do so, we will build a cycle $\gamma$ containing $p$
vertices, such that $D^M[V(D^M) \setminus V(\gamma)]$ contains a
spanning cycle denoted by $\gamma'$. Let $s$ (\emph{resp.} $s'$) be
the number of vertices in the path $P$ (resp. $P'$) along $C^M$ from
$y_\ell$ to $z_1$ (\emph{resp.}  from $z_\ell$ to $y_1$). We have
$s+s'=p$, thus either we have $s\leq p/2$ or $s'\leq p/2$. We will
suppose that the former holds, since an analogous reasoning can be
applied for the other case. In the following, we will denote by $i_0$
the smallest index $j$ such that $s + \sum^j_{i=1}{c_i}>p$. Such index
exists, since $s < p$ and $s + \sum^\ell_{i=1}{c_i} >
\sum^\ell_{i=1}{c_i}=2k-p\ge p$. The cycle $\gamma$ (resp. $\gamma'$)
will be obtained as the union of the path $P$ (resp. $P'$) and a path
$Q$ (resp. $Q'$) well chosen in $C_1^M\dots C_\ell^M$. To design $Q$
and $Q'$ we consider several cases.\medskip

 \begin{figure}
\centering \includegraphics[width=0.8\textwidth]{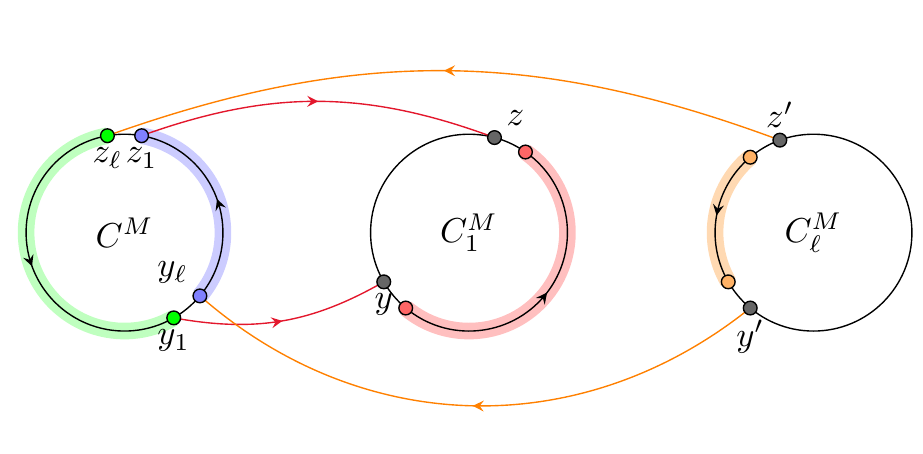}
  \caption{An illustrative case of the proof of
   Claim~\ref{claim:pathlenght}. The red path contains $i$ vertices and the orange one $i'$ vertices.}
  \label{fig:pathlenght}
 \end{figure}
First, assume that we have $1<i_0<\ell$. According to
Claim~\ref{claim:pathlenght} applied with $i=0$ and $i'=0$ there
exists a pair of vertices $\{y, z\}$ of $C_1^M$ such that $y$ is the
successor of $z$ along $C_1^M$ and with $y_1z, z_1y \in
A(D^M)$. Similarly, there is $\{y', z'\}$ in $C_\ell^M$ such $y'$ is
the successor of $z'$ along $C_\ell^M$ and $z'z_\ell, y'y_\ell \in
A(D^M)$. As $C_i^M$ dominates $C_j^M$ for any $i<j$, we consider in
$D^M$ the path $Q$ starting in $y$, containing every vertices of
$C_1^M$ except $z$, every vertices of $C_j^M,$ for any $ 2\leq j \leq
i_0 -1$ and $p - s - \sum^{i_0-1}_{i=1}{c_i}$ consecutive vertices of
$C_{i_0}^M$ and finally ending with the vertex $y'$. Similarly we
construct the path $Q'$ containing $z$, the remaining vertices of
$C_{i_0}^M$, every vertices of $C_j^M,$ for any $ i_0< j < \ell$ and
every vertices of $C_\ell^M$ except $y'$, that is $Q'$ ends in
$z'$. As $z_1y$, $y'y_\ell$, $y_1z$ and $z'z_\ell$ are arcs of $D^M$
$\gamma=P\cup Q$ and $\gamma'=P'\cup Q'$ are cycles and they form a
2-cycle-factor of $D^M$. To conclude this case, it remains to notice
that the number of vertices in $\gamma$ is \mbox{$s+(c_1 -1) +
  (\sum^{i_0-1}_{i=2}{c_i}) + (p - s - \sum^{i_0-1}_{i=1}{c_i}) + 1 =
  p$}.

\medskip

In the case where $i_0 = 1$, Claim~\ref{claim:pathlenght} applied with
$i=p-s-2$ and $i'=0$ asserts that there exist $\{y, z\}$ in $C_1^M$
such that there are $p-s-2$ vertices from $y$ to $z$ along $C_1^M$
with $y_1z, z_1y \in A(D^M)$. As we assume that $p\ge 3$ and we have
$s\le p/2$, we know that $p-s\ge 2$. There also are $\{y', z'\}$ in
$C_\ell^M$ such that $y'$ is the successor of $z'$ along $C_\ell^M$
and $z'z_\ell, y'y_\ell \in A(D^M)$.  Thus we construct $Q$ starting
from $y$, containing $p-s-1$ vertices of $C_1^M$ and ending in
$y'$. The path $Q'$ starts in $z$, contains the remaining vertices of
$C_1^M$, every vertices of $C_j^M,$ for any $ 1< j < \ell$ and every
vertices of $C_\ell^M$ except $y'$. That is, $Q'$ ends in $z'$. As
previously, we easily check that $\gamma=P\cup Q$ and $\gamma'=P'\cup
Q'$ form a 2-cycle-factor of $D^M$ of lengths $p$ and $2k-p$.\par
\medskip
The case $i_0 = \ell$ is symmetric to the previous one.
\medskip

To check the last part of the statement, we have to guarantee that
$D[C']$ is not isomorphic to $F_{2p}$, where $C'$ denote the cycle of
$D$ corresponding to $\gamma$ (i.e. such that $C'^M=\gamma$). To do
so, notice that, in all cases, we added $y$ and $y'$ to $P$ in order
to close $\gamma$. In $D^M$, as $y\in C_1^M$ and $y'\in C_\ell^M$, we
obtain that $yy'$ is an arc of $D^M$ and $y'y$ is not an arc of
$D^M$. Then $C'$ contains four vertices $y$, $C'(y) , y'$ and $C'(y')$
such that $yC(y)$, $C'(y)y'$, $y'C'(y')$ and $yy'$ are arcs of
$D$. However $F_{2p}$ does not contain such a subdigraph. So $D[C']$
is not isomorphic to $F_{2p}$.
\end{proof} 

\section{Proof of Theorem~\ref{theo:2cf}}
\label{sec:thm}

We prove a slightly stronger version of Theorem~\ref{theo:2cf} where
we ask for the first cycle of the cycle-factor to be different from
$F_{2p}$ if $p$ is even (notice that $F_{2p}$ is not defined for odd
$p$). Namely, we prove the following result.

\begin{theorem}%
 \label{theo:2cfbis}
 For $k\ge 3$ let \(D\) be a \(k\)-regular bipartite tournament not
 isomorphic to \(F_{4k}\). Then for every \(p\) with \(3\le p\le k\),
 \(D\) has a 2-cycle-factor \((C_1,C_2)\) where $C_1$ has length
 \(2p\) and if $p$ is even, $C_1$ is not isomorphic to $F_{2p}$.
\end{theorem}

We prove this statement by induction on $p$.  By the result of Bai, Li
and He~\cite{BLH14} the statement is true for $p=3$ and the basis case
for the induction holds.  So for $3\leq p< k$ we consider $D=(V,A)$ a
$k$-regular bipartite tournament which admits a
$(2p,4k-2p)$-cycle-factor $(C_1,C_2)$ where $C_1$ is not isomorphic to
$F_{2p}$ if $p$ is even. In particular, notice that $D$ is not
isomorphic to $F_{4k}$.  We want to show that $D$ admits a
$(2(p+1),4k-2(p+1))$-cycle-factor whose first cycle is not isomorphic
to $F_{2(p+1)}$ if $p+1$ is even. In the following we call a {\it good
  cycle-factor} such a cycle-factor.\\ We denote by $(S,T)$ the
bipartition of $D$ and by $(C_1,C_2)$ the $(2p,4k-2p)$-cycle-factor of
$D$, with $D[C_1]$ not isomorphic to $F_{2p}$. We also denote by $M_u$
the arcs of $C_1\cup C_2$ going (up) from $S$ to $T$ and by $M_d$ the
arcs of $C_1\cup C_2$ going (down) from $T$ to $S$. It is clear that
$M_u$ ans $M_d$ are perfect matchings of $D$ and that their union is
$C_1\cup C_2$. For $M$ being either $M_u$ or $M_d$, the digraph $D^M$
admits the 2-cycle-factor $(C_1^M,C_2^M)$ with $|C_1^M|=p$ and
$|C_2^M|=2k-p$. Notice that, for even $p$, having $C_1$ not isomorphic
to $F_{2p}$ is equivalent to having $D^M[C_1^M]$ being not isomorphic
to the balanced complete bipartite digraph on $p$ vertices.\\ To form
a good cycle-factor from $(C_1^M,C_2^M)$, we will have a case-by-case
study according to the structure of the non-arc in the digraph
$D^M$. Prior to this study, we introduce some needed tools.

\subsection{Switch along an anti-cycle}

In this subsection, we first consider a matching $M$ of the
$k$-regular bipartite tournament $D$ made from arcs from $S$ to $T$
and we define an operation allowing some local change in $M$.

\begin{lemma}%
 \label{lem:switch}
If \(D^M\) contains an anti-cycle \(u_1,\dots ,u_t\) with \(t\ge 2\),
then the set \(M'\) of arcs defined in $D$ by \(M'=(M\setminus
\bigcup_{i=1}^{t} u_{i}M(u_i)) \cup (\bigcup_{i=1}^{t-1}
u_{i+1}M(u_{i}) \cup u_1M(u_t))\) is a perfect matching of
\(D\). Moreover, for every \(v\notin \lb{}u_1,\dots ,u_t\rb{}\), we
have \(N^+_{D^{M'}}(v)=N^+_{D^{M}}(v)\) and, for every \(i\) with \( 2
\leq i \leq t\), we have
\mbox{\(N^+_{D^{M'}}(u_i)=N^+_{D^{M}}(u_{i-1})\)} and
\(N^+_{D^{M'}}(u_1)=N^+_{D^{M}}(u_{t})\).
\end{lemma}

\begin{proof} If \(u_1,\dots ,u_t\) is an anti-cycle of \(D^M\) it follows
  by definition that \(u_{i+1}M(u_i)\) is an arc of \(D\) for every
  \(i\) such that \(1 \leq i \leq t-1\) as well as
  \(u_1M(u_t)\). Thus, \(M'\) is a perfect matching of \(D\). For
  every \(i\) with \(1 \leq i \leq t-1\) we have
  \(M'(u_{i+1})=M(u_i)\) in \(D\) (and \(M'(u_1)=M(u_t)\)), then
  \(N^+_{D^{M'}}(u_{i+1})=N^+_{D^{M}}(u_{i})\) (and
  \(N^+_{D^{M'}}(u_{1})=N^+_{D^{M}}(u_{t})\)). The out-neighborhood
  of the other vertices are unchanged.
\end{proof}

The ``shifting'' operation between matchings \(M\) and \(M'\)
described in the previous lemma is called a \emph{switch along the
  anti-cycle} \(u_1,\dots,u_t\). The first easy observation we can
make on the new contracted digraph is the following.

\begin{corollary}%
 \label{cor:switch-cycle-factor}
 If \(D^M\) has a cycle-factor and contains an anti-cycle \(u_1,\dots
 ,u_t\) with \(t\ge 2\), then the digraph obtained after the switch
 along the anti-cycle \(u_1,\dots ,u_t\) has a cycle-factor.
\end{corollary}

\begin{proof}
 Let \(C\) be the anti-cycle \(u_1,\dots, u_t\) and let \(M'\) be the
 perfect matching obtained after the switch along \(C\). Moreover, let
 \(\Delta{}\) be the subdigraph induced by the arcs of the cycle-factor in \(D^M\), and let \(\Delta'\) be the subdigraph induced by
 the switch of \(\Delta{}\) along \(C\). By Lemma~\ref{lem:switch}, we
 make a cyclic permutation on the out-neighborhoods of the vertices of
 \(C\). So for every vertex \(x\) in \(S\) we have
 \(d^+_{\Delta'}(x)=d^+_{\Delta}(x)=1\) and
 \(d^-_{\Delta'}(x)=d^-_{\Delta}(x)=1\). Therefore \(\Delta'\) is a
 cycle-factor of \(D^{M'}\).
\end{proof}

Claim~\ref{claim:good-anti-cycle} below is our main application of a
switch along an anti-cycle in a contracted digraph. Before stating it,
we need the following simple result.

\begin{lemma}
\label{lem:SommeCF}
If $D^M$ contains a cycle-factor $\{B_1,\dots , B_l, B_{l+1}, \dots
B_{l'}\}$ such that $|B_1|+\dots +|B_l|=p+1$ and $D^M[B_1\cup \dots
  \cup B_l]$ is strongly connected and not isomorphic to a balanced
complete bipartite digraph, then $D$ admits a good cycle-factor.
\end{lemma}

\begin{proof}
For $i=1,\dots ,l'$ we denote by $\tilde{B}_i$ the cycle of $D$ such
that $\tilde{B}_i^M=B_i$.  Since $D^M[B_1\cup \dots \cup B_l]$ is
strongly connected, the digraph $D[\tilde{B}_1\cup \dots \cup
  \tilde{B}_l]$ is also strongly connected and admits a
cycle-factor. So, by Theorem~\ref{theo:haggvist-manoussakis}, it has a
Hamiltonian cycle $C$ which is of length $2p+2$.  Moreover, as
$D^M[B_1\cup \dots \cup B_l]$ is not isomorphic to a balanced complete
bipartite digraph, then $D[C]$ is not isomorphic to $F_{2p+2}$ if $p$
is odd.  Using Lemma~\ref{lem:cf-to-2cf} on the cycle-factor $\{C,
\tilde{B}_{l+1}, \dots, \tilde{B}_{l'}\}$, it proves that $D$ contains
a good cycle-factor.
\end{proof}

Now, back to the proof of Theorem~\ref{theo:2cfbis}, we see a first
case where it is possible to extend the cycle $C_1$ to obtain a good
cycle-factor. Recall that the digraph $D^{M}$ contains the
2-cycle-factor $(C_1^{M},C_2^{M})$, where here again $M$ stands for
the matching $M_u$ or the matching $M_d$ of $D$.

\newtheorem{claimt}{Claim}[theorem]

\begin{claimt}
 \label{claim:good-anti-cycle}
 If $D^{M}$ contains an anti-cycle $H$ such that $H\setminus
 V(C_1^{M})$ is an anti-path from $x$ to $y$ with $x = C_2^{M}(y)$
 then $D$ admits a good cycle-factor.
\end{claimt}

\begin{proof}
To shorten notations, we denote $C_1^{M}$ by $C$ and
$C_2^{M}$ by $C'$.
Let $H=a_1\dots a_t$ be an anti-cycle of $D^{M}$ such
 that $a_1\dots a_s$ is an anti-path of $D^{M}[C']$, $a_1$ is the
 successor of $a_s$ along $C'$ and $a_{s+1}\dots a_t$ is an anti-path
 of $D^{M}[C]$. Moreover, for every $i$ with $1 \leq i \leq t$ we denote by $b_i$ the out-neighbor of $a_i$ along $C$ or $C'$.  An illustrative case
 is depicted in Figure~\ref{fig:switch-strong}.
 
 \begin{figure}[!ht]
  \centering \includegraphics[scale=0.9]{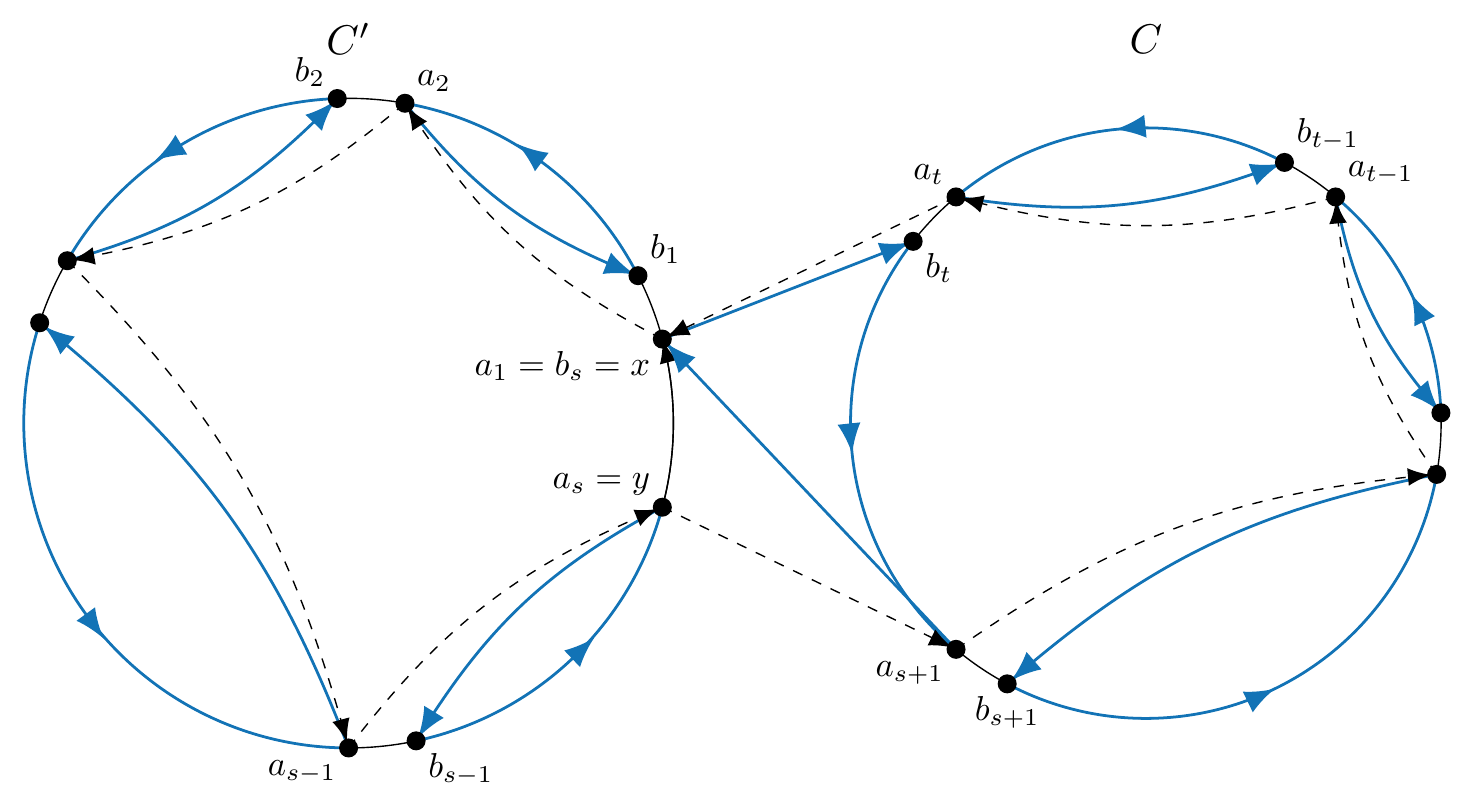}
  \caption{An illustrative case of the proof of
   Claim~\ref{claim:good-anti-cycle}. The dashed arcs form the anti-cycle
   $H$ and the blue arcs form the cycle-factor $B$ of $D^{{M}'}$.}
  \label{fig:switch-strong}
 \end{figure}
 
 We perform a switch exchange on $H$ to obtain the digraph $D^{{M}'}$. By
 Corollary~\ref{cor:switch-cycle-factor}, the digraph $D^{{M}'}$ contains
 a cycle-factor. More precisely, we pay attention to $B$ the
 cycle-factor derived from $C\cup C'$. By Lemma~\ref{lem:switch}, the
 out-neighbor in $B$ of every vertex not in $\{a_1,\dots ,a_t\}$ is its
 out-neighbor in $C\cup C'$ and the out-neighbor in $B$ of every
 vertex $a_i$ in $\{a_1,\dots ,a_t\}$ is $b_{i-1}$ (where indices are
 given modulo $t$). As there is only one arc of $H$ from $C$ to $C'$
 and one arc of $H$ from $C'$ to $C$, the only arc of $B$ from $V(C)$
 to $V(C')$ of $B$ is $a_{s+1}b_s$ and the only arc of $B$ from $V(C')$
 to $V(C)$ is $a_1b_t$. So $B$ contains a subset $B_1$ of cycles
 covering $V(C)\cup \{a_1\}$ and a subset $B_2$ of cycles covering
 $V(C')\setminus \{a_1\}$. Thus the cycles of $B_1$ (resp. $B_2$) form a
 cycle-factor of $D^{{M}'}[C\cup \{a_1\}]$ (resp. $D^{{M}'}[C'\setminus \{a_1\}]$)
 and we denote by $\tilde{B_1}$ (resp. $\tilde{B_2}$) the corresponding
 cycle-factors of $D$. Moreover, for $i=s+1,\dots t$ the arcs
 $a_i{M}(a_i)$ belongs to $D$ and as ${M}(a_i)={M}'(a_{i+1})$ they link the
 cycles of $\tilde{B_2}$ in a strongly connected way. Thus
 $D[\tilde{B_2}]$ is strongly connected and so by
 Theorem~\ref{theo:haggvist-manoussakis} it has an hamiltonian cycle
 $B_3$ on $2(p+1)$ vertices. In addition if $p+1$ is even $D[B_3]$ is not
 isomorphic to $F_{2(p+1)}$ as it contains $C_1$ as a subdigraph, $C_1$
 being a cycle on $2p$ vertices with $p$ odd.  Therefore
 $\tilde{B_1}\cup B_3$ forms a cycle-factor of $D$ with a cycle, $D[B_3]$,
 of length $2(p+1)$ not isomorphic to $F_{2(p+1)}$ and we can conclude
 with Lemma~\ref{lem:cf-to-2cf}.
\end{proof}

The next claim is an easy case where we can insert a vertex of
$C_2^{M}$ into $C_1^{M}$.

\begin{claimt}
 \label{claim:extendingC}
 If $C_2^{M}$ contains three consecutive vertices $a$, $b$ and $c$ (in
 this order along $C_2^M$) and $C_1^{M}$ contains two consecutive
 vertices $x$ and $y$ (in this order along $C_1^M$) such that $ac$,
 $xb$ and $by$ are arcs of $D^{M}$ then $D$ admits a good
 cycle-factor.
\end{claimt}

\begin{proof}
 It is clear that using the arcs $ac$, $xb$ and $by$ we can form a
 2-cycle-factor of $D^{M}$, with one cycle of length $2k-(p+1)$
 covering $C_2^M\setminus \{b\}$ and the other of length $p+1$
 covering $C_1^{M}\cup \{b\}$. Let us denote by $\tilde{C}$ this
 latter one. If $p+1$ is even notice that the cycle of $D$
 corresponding to $\tilde{C}$ cannot be isomorphic to
 $F_{2(p+1)}$. Indeed otherwise $\tilde{C}$ would be isomorphic to a
 complete bipartite digraph but $\tilde{C}$ contains $C$ has a
 subdigraph which is a cycle on $p$ vertices with $p$ odd, a
 contradiction.
\end{proof}

Now we can prove the following claim, that we will intensively use in
the remaining of the proof of Theorem~\ref{theo:2cfbis}.

\begin{claim}
 \label{claim:external-arc}
 Assume that $\overline{D^{M}[C_2^{M}]}$ is not strongly connected and
 denote by $S_1,\dots ,S_l$ its strongly connected components.  If
 there exists an arc $ab$ of $C_2^{M}$ such that there is an anti-path
 in $D^M[C_2^{M}]$ from $b$ to $a$ and $a\in S_i$, $b\in S_j$ for
 $i\neq j$, then $D$ admits a good cycle-factor.
\end{claim}

\begin{proof}
In $D^{M}$, we denote $C_2^{M}$ by $C'$ and $C_1^{M}$ by $C$. All the
proof stands in $D^{M}$.  First assume that there exists a vertex
$c\in C$ such that $ac$ and $cb$ are anti-arcs, then the anti-cycle
formed by the anti-path from $b$ to $a$ in $C'$ completed with the
anti-arcs $ac$ and $cb$ satisfies the hypothesis of
Claim~\ref{claim:good-anti-cycle} and we can conclude.\\ Hence, we
assume that $\overline{N^+_C}(a)\cap
\overline{N^-_C}(b)=\emptyset$. In particular, we have $
\overline{d^+_{C}}(a) + \overline{d^-_{C}}(b) \le p$. Let us denote by
$A$ the set of all the vertices of $C'$ for whom there is a anti-path
from $a$ to them, and by $B$ the set $C'\setminus A$.  The set $A$
dominates the set $B$, we have $|A|+|B|=2k-p$ and also $S_i\subseteq
A$ and $S_j\subseteq B$ (otherwise $a$ and $b$ would have been in the
same connected component of $\overline{D^{M}[C']}$). Therefore, we
have
 $$2k-2=\overline{d^+}(a) + \overline{d^-}(b)\le (
\overline{d^+_{C}}(a) + \overline{d^-_{C}}(b)) +
\overline{d^+_{C'}}(a) + \overline{d^-_{C'}}(b)\le p+
|A|-1+|B|-1=2k-2$$ and thus we have equalities everywhere. In
particular we obtain that $(A,B)$ is a partition of $V(C')$ with
$A\setminus \{a\}\subseteq \overline{N^+}(a)$ and $B\setminus
\{b\}\subseteq \overline{N^-}(b)$. As a consequence for every $x_A\in
A$ and $x_B\in B$ there exists an anti-path from $x_B$ to $x_A$.
Another consequence is that $\overline{d^+_{C}}(a) +
\overline{d^-_{C}}(b)=p$ and that $C$ admits a partition into
$\overline{N^+_C}(a)$ and $\overline{N^-_C}(b)$.\\ So let $b'$ be the
successor of $b$ along $C'$, that is $b'=C'(b)$, and assume first that
$b'\in B$. Hence, for every $x\in \overline{N^+_C}(a)$ the arc
$xb$ exists in $D^M$ (as we assume that $\overline{N^+_C}(a)\cap
\overline{N^-_C}(b)=\emptyset$) and so $bC(x)$ is an anti-arc,
otherwise we can insert $b$ into $C$ and shortcut the path $abb'$
using Claim~\ref{claim:extendingC}. Thus we have
$C(\overline{N^+_C}(a)) \subseteq \overline{N^+_C}(b)$ and In
particular, we obtain $\overline{d^+_C}(a)\le \overline{d^+_C}(b)$.
Hence, we have
    \begin{equation*}
   \begin{split}
    2k-2 = \overline{d^+}(a)+\overline{d^-}(b')
    & =\overline{d^+_C}(a)+|A|-1+
  \overline{d^-_C}(b')+ \overline{d^-_B}(b')\\
    & \le \overline{d^+_C}(b)+
  \overline{d^-_C}(b')+|A|-1+|B|-2
   \end{split}
  \end{equation*}
  and then $\overline{d^+_C}(b)+
  \overline{d^-_C}(b')\ge p+1$.
  In particular, there exists $c\in C$ such
 that $bc$ and $cb'$ are anti-arcs and we can apply
 Claim~\ref{claim:good-anti-cycle} to the anti-cycle $bb'c$ to obtain a
 good cycle-factor of $D$.\\ Now, we assume that we have $b'\in A$. And
 more generally we can assume that $C'$ has no two consecutive vertices
 lying in $B$. Indeed, otherwise considering a non empty path of
 $C'[B]$, $vw$ its first arc and $u$ the predecessor of $v$ along $C'$
 we can proceed as before with $a=u$, $b=v$ and $b'=w$. Symmetrically,
 $C'$ has no two consecutive vertices in $A$ and $C'$ alternates
 between $A$ and $B$. In particular, $p$ is even and we write
 $C'=a_1b_1a_2b_2\dots a_{k-p/2}b_{k-p/2}$ with $A=\{a_1,\dots
  ,a_{k-p/2}\}$ and $B=\{b_1,\dots ,b_{k-p/2}\}$ (indices in $C'$ will
 be given modulo $k-p/2$). By the above arguments we also have
 $A\setminus \{a_i\} \subseteq \overline{N^+}(a_i)$ and $B\setminus
  \{b_i\} \subseteq \overline{N^-}(b_i)$ for every $1\le i\le k-p/2$ implying that $A$ and $B$ are
 two independent sets of $D^{M}$.\par
 Notice that, $p$ being even, we have $p\ge 4$ and $k-p/2>p/2\ge
 2$. So we obtain that $k-p/2\ge 3$.  Now, for every $i$ with $1\leq i
 \leq k - p/2$, we denote by $B_i$ the set $\overline{N^-_C}(b_i)$ and
 by $A_i$ the set $\overline{N^+_C}(a_i)$.  As $a_ib_i$ is an arc of
 $C'$, as there exists an anti-path from $b_i$ to $a_i$ in $D[C']$ and
 as $a_i$ and $b_i$ do not belong to the same strongly connected
 component of $\overline{D^{M}[C']}$ (because $A$ dominates $B$), we
 can argue as before with $a_i$ playing the role of $a$ and $b_i$ the
 one of $b$. In particular, we obtain that $(A_i,B_i)$ is a partition
 of $V(C)$ for every $i\in \{1,\dots k-p/2\}$ with $|A_i|=|B_i|=p/2$.
 Moreover, assume that there exists $x\in C$ and $i\in \{1,\dots
 k-p/2\}$ such that $a_{i}x$ and $xb_{i+1}$ are anti-arcs. In this
 case, we modify $C'$ into the cycle $C''$ by replacing the subpath
 $a_{i-1}b_{i-1}a_ib_ia_{i+1}b_{i+1}$ of $C'$ by
 $a_{i-1}b_ia_{i+1}b_{i-1}a_ib_{i+1}$. Now, $a_ib_{i+1}$ is an arc of
 $C''$ and there exists an anti-path $P$ in $D^{M}[C'']$ from
 $b_{i+1}$ to $a_i$. So we can conclude by applying
 Claim~\ref{claim:good-anti-cycle} to the anti-cycle $Px$. So it means
 that $A_i\cap B_{i+1}=\emptyset$ for every $i\in \{1,\dots
 k-p/2\}$. We deduce then that all the $A_i$ coincide as well as all
 the $B_i$ and in particular we have $A$ anti-dominates $A_i$ and
 $A_i$ dominates $B$ for every $i\in \{1,\dots k-p/2\}$.\par\smallskip
  To conclude, consider $s,t \in [1, k-p/2]$ such that $b_sa_t$ is an
  anti-arc (such an anti-arc exists since there exists an anti-path
  from $b$ to $A$). If there exists an anti-arc from $x_a\in A_t$ to
  $x_b\in B_t$, then we conclude with Claim~\ref{claim:good-anti-cycle}
  applied to the anti-cycle $x_ax_bb_tb_sa_t$. Otherwise it means that
  $A_t$ dominates $B_t$ and so as $A_t$ dominates $ B$, we have $A_t$
  anti-dominates $A$. But then we obtain $\{b_s\}\cup A_t\cup
  A\setminus \{a_t\} \subseteq \overline{N^-}(a_t)$ and
  $\overline{d^-}(a_t)\ge k$, a contradiction.
\end{proof}

\subsection{Properties of the 2-cycle-factor $(C_1,C_2)$.}

Now, we define four different properties of the 2-cycle-factor
$(C_1,C_2)$ of $D$. For this, we first need the next claim.

\begin{claim}
 \label{claim:one-cfc}
  Let $C$ be $C_1$ or $C_2$.  We have: 
  \begin{itemize}
  	\item either $\overline{D^{M_u}[C^{M_u}]}$ or
  	$\overline{D^{M_d}[C^{M_d}]}$ has exactly one terminal strong
  	component
  	\item or $|C|$ is congruent to 0 modulo 4 and $D[C]$ is
  	isomorphic to $F_{|C|}$.
  \end{itemize}  

 Similarly, we have: 
 \begin{itemize}
 	\item either
 	$\overline{D^{M_u}[C^{M_u}]}$ or $\overline{D^{M_d}[C^{M_d}]}$ has
 	exactly one initial strong component
 	\item  	or $|C|$ is congruent to 0
 	modulo 4 and $D[C]$ is isomorphic to $F_{|C|}$
 \end{itemize} 
\end{claim}

\begin{proof}
We prove the statement for terminal strong components. The proof for
initial ones is similar.\\ We denote by $D'$ the digraph $D[C]$ with
bipartition $(S',T')$ where $S'=S\cap V(C)$ and $T'=T\cap V(C)$.  So,
assume that $\overline{D'^{M_u}}$ has at least two terminal strong
components and denote by $A$ and $B$ two such components. In
$D'^{M_u}$, it means that $A$ dominates $D'^{M_u}\setminus A$ and that
$B$ dominates $D'^{M_u}\setminus B$. In $D'$, by considering that $A$
and $B$ are subsets of $S'$, we then have $C(A)$, the successors of
the vertices of $A$ along $C$, dominates $S'\setminus A$ and $C(B)$
dominates $S'\setminus B$.  Notice that $A$ and $B$ are two disjoint
sets of $S'$ and that $C(A)$ and $C(B)$ are two disjoint sets of $T'$
with $|C(A)|=|A|$ and $|C(B)|=|B|$.\\ Now, assume that there exists a
vertex $u$ of $C(A)$ such that $u'=C(u)$ does not belong to $B$. Then
let $v$ be a vertex of $C(B)$. As $u'\notin B$, the arc $vu'$ belongs
to $D'$ and then there exists an anti-arc from $u$ to $v$ in
$D'^{M_d}$.  Moreover, for every vertex $w$ of $T'$, if $C(w)$ does
not belongs to $A$ then there is an anti-arc in $D'^{M_d}$ from $w$ to
$u$ and if it belongs to $A$, then there is an anti-arc in $D'^{M_d}$
from $w$ to $v$. Thus, $\overline{D'^{M_d}}$ has exactly one initial
strong component, containing $v$.\\ Otherwise, in $D'$ every vertex
$u$ of $C(A)$ satisfies $C(u)\in B$ and similarly every vertex $v$ of
$C(B)$ satisfies $C(v)\in A$. In particular, we have $|C(A)|\le
|B|=|C(B)| \le |A|=|C(A)|$ and so $|A|=|B|=|C(A)|=|C(B)|$. Moreover,
as $C$ contains all the vertices of $D'$, $(A,B)$ is a partition of
$S'$ and $(C(A),C(B))$ is a partition of $T'$ with $C(A)$ dominates
$B$ and $C(B)$ dominates $A$. In particular, $|C|$ is congruent to 0
modulo 4. As all the arcs of $C$ from $T'$ to $S'$ go from $C(A)$ to
$B$ or from $C(B)$ to $A$ and that $C(A)$ dominates $B$ and $C(B)$
dominates $A$ in $D'$, the sets $C(A)$ and $C(B)$ respectively induce
an independent set in $D'^{M_d}$, that is a complete digraph in
$\overline{D'^{M_d}}$. As they form a partition of the vertex set of
$\overline{D'^{M_d}}$, it has exactly one initial strong component
except if there is no arc between $C(A)$ and $C(B)$. In this later
case, it means that all the arcs from $A$ to $C(A)$ and from $B$ to
$C(B)$ are contained in $D'$.  But, then $D'[C]$ is isomorphic to
$F_{|C|}$.
\end{proof}

In the case where $\overline{D^{M_u}[C_2^{M_u}]}$ has exactly one
initial strong component, we say that the 2-cycle-factor $(C_1,C_2)$
has property $\boldsymbol{{\cal Q}_{\rm up}}$. And in the case where
$\overline{D^{M_d}[C_2^{M_d}]}$ has exactly one initial strong
component, we say that $(C_1,C_2)$ has property $\boldsymbol{{\cal Q}_{\rm
  down}}$. By Claim~\ref{claim:one-cfc}, we know that either
$(C_1,C_2)$ has property ${\cal Q}_{\rm up}$ or ${\cal Q}_{\rm down}$
or that $|C_2|$ is congruent to 0 modulo 4 and that $D[C_2]$ is
isomorphic to $F_{|C_2|}$.\\

Now, let us define another pair of properties for the 2-cycle-factor
$(C_1,C_2)$.  As $D$ is a bipartite tournament every vertex of $C_2$
has an in-neighbor or an out-neighbor in $C_1$. Moreover, as $D$ is
$k$-regular and $|C_1|=2p<2k$ there exists at least one arc from $C_1$
to $C_2$ and one arc from $C_2$ to $C_1$. So, it is easy to check that
there exists an arc $uv$ of $C_2$ such that $N^-_{C_1}(u)\neq
\emptyset$ and that $N^+_{C_1}(v)\neq \emptyset$.  If we have $u\in S$
and $v\in T$ we say that $(C_1,C_2)$ has property $\boldsymbol{{\cal
    P}_{\rm down}}$. In this case, it means that the arc $uv$ is an
arc of $M_u$, and in $D^{M_u}$ the vertex $u$ has an in-neighbor and
an out-neighbor in $C_1^{M_u}$. Otherwise, that is when we have $v\in
S$ and $u\in T$, we say that $(C_1,C_2)$ has property
$\boldsymbol{{\cal P}_{\rm up}}$. In this case, let $u'$ be the
predecessor of $u$ along $C_2$, that is $u = C_2(u')$. So, in
$D^{M_u}$, $u'v$ is an arc of $C_2^{M_u}$, the vertex $u'$ has an
anti-out-neighbor in $C_1^{M_u}$ and $v$ has an anti-in-neighbor in
$C_1^{M_u}$.

\begin{claim}
\label{claim:not-Pdown}
If $(C_1,C_2)$ does not satisfies ${\cal P}_{\rm down}$, then in
$D^{M_u}$ every vertex of $C_2^{M_u}$ anti-dominates $C_1^{M_u}$ or is
anti-dominated by $C_1^{M_u}$.
\end{claim}

\begin{proof}
Let $x$ be a vertex of $C_2^{M_u}$. Since $D$ does not satisfies
${\cal P}_{\rm down}$, then in $D$, either $C_2(x)$ is dominated by
$C_1\cap S$ or $x$ dominates $C_1\cap T$. In the first case, it means
that there is no arc from $x$ to $C_1^{M_u}$ in $D^{M_u}$, while in
the latter, it means that there is no arc from $C_1^{M_u}$ to $x$.
\end{proof}

Notice that if $(C_1,C_2)$ satisfies property ${\cal Q}_{\rm up}$,
then exchanging the role of $S$ and $T$ in $D$ (and then of $M_u$ and
$M_d$) leads to $(C_1,C_2)$ satisfies property ${\cal Q}_{\rm down}$,
and conversely. We also have the similar property with ${\cal P}_{\rm
  up}$ and ${\cal P}_{\rm down}$.

Thus, without loss of generality, we assume that $(C_1,C_2)$ has the
property ${\cal P}_{\rm up}$. Then, we study, in this order, the three
different cases: either $(C_1,C_2)$ has property ${\cal Q}_{\rm up}$,
or $D[C_2]$ is isomorphic to $F_{|C_2|}$, or $(C_1,C_2)$ satisfies
property ${\cal Q}_{\rm down}$.

\subsection{Case A: $(C_1,C_2)$ has property ${\cal Q}_{\rm up}$}

So, we know that $\overline{D^{M_u}[C_2^{M_u}]}$ has exactly one
initial strong component. If it is not strong itself, there exists an
arc of $C_2^{M_u}$ entering into its unique initial strong component
and we can directly conclude with
Claim~\ref{claim:external-arc}. Then, we assume that
$\overline{D^{M_u}[C_2^{M_u}]}$ is strongly connected.  We consider
several cases.\medskip

\noindent
{\bf Case 1 : $\overline{D^{M_u}[C_1^{M_u}]}$ is strongly connected
too}. As $(C_1,C_2)$ has the property ${\cal P}_{\rm up}$, there
exist $u$ and $v$ in $D^{M_u}$ such that $v$ is the successor of $u$
along $C_2^{M_u}$, $u$ has an anti-out-neighbor $u'$ in $C_1^{M_u}$
and $v$ has an anti-in-neighbor $v'$ in $C_1^{M_u}$. As both
$\overline{D^{M_u}[C_1^{M_u}]}$ and $\overline{D^{M_u}[C_2^{M_u}]}$
are strongly connected, there exists an anti-path from $v$ to $u$ in
$D^{M_u}[C_2^{M_u}]$ and an anti-path from $u'$ to $v'$ in
$D^{M_u}[C_1^{M_u}]$. So, we can form an anti-cycle in $D^{M_u}$ which
satisfies the hypothesis of Claim~\ref{claim:good-anti-cycle} and
conclude that $D$ contains a good cycle-factor.\medskip

\noindent
{\bf Case 2 : $\overline{D^{M_u}[C_1^{M_u}]}$ is not strongly
  connected.}  In what follows, to shorten the notation, we denote
$C_1^{M_u}$ by $C$ and $C_2^{M_u}$ by $C'$.  So, as
$\overline{D^{M_u}[C]}$ is not strong, there exists a partition
$(A,B)$ of $V(C)$ such that there is no anti-arcs from $A$ to $B$,
thus we have $A$ dominating $B$ in $D^{M_u}$.\smallskip

\noindent
{\bf Case 2.1 :} {\it There exist two vertices $a \in A$ and $b\in B$
 such that there is an anti-path $P$ from $b$ to $a$ in
 $D^{M_u}[C]$.} So, in $D^{M_u}$ we have $$\overline{d^+}(a) +\overline{d^-}(b) =
  2k-2$$ Since $\overline{d^+}(a) \leq |A|-1+\overline{d^+_{C'}}(a)$ and
 $\overline{d^-}(b) \leq |B|-1+\overline{d^-_{C'}}(b)$ and $|A|+|B|=p$,
 we get
 $$ \overline{d^+_{C'}}(a) + \overline{d^-_{C'}}(b) \geq 2k-p$$

 If $ \overline{d^+_{C'}}(a) + \overline{d^-_{C'}}(b) > 2k-p=|C'|$ then
 there exist two vertices $a'$ and $b'$ in $C'$ such that $aa'$ and
 $bb'$ are anti-arcs, and $b'$ is the successor of $a'$ along $C'$. As
 $\overline{D^{M_u}[C']}$ is strongly connected, there exists an
 anti-path $Q$ from $a'$ to $b'$ in $D^{M_u}[C']$. So the concatenation of the paths $P$, $Q$ and the arcs $a'a$ and $bb'$ forms
 an anti-cycle of $D^{M_u}$ satisfying the conditions of
 Claim~\ref{claim:good-anti-cycle} and we can conclude that $D$
 contains a good cycle-factor. See Figure~\ref{fig:casa21} which depicts this subcase. 
 
\begin{figure}
	\centering
	\includegraphics[width=0.8\linewidth]{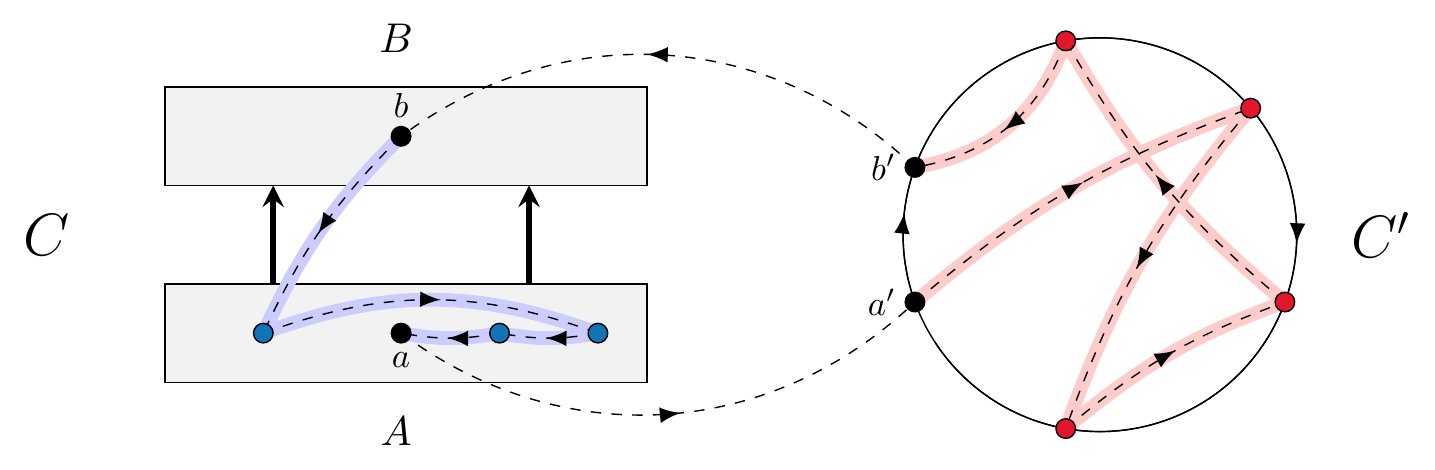}
	\caption{Illustration of the anti-cycle when $\overline{D^{M_u}[C]}$ is not strongly
		connected but there exist two vertices $a \in A$ and $b\in B$
		such that there is an anti-path $P$ from $b$ to $a$ in
		$D^{M_u}[C]$ (depicted in blue) and, there are two vertices $a'$ and $b'$ in $C'$ such that $aa'$ and
		$bb'$ are anti-arcs, and $b'$ is the successor of $a'$ along $C'$. We can use the red anti-path from $a'$ to $b'$ and $P$ to form an anti-cycle.}
	\label{fig:casa21}
\end{figure}

 So, we assume that $ \overline{d^+_{C'}}(a) + \overline{d^-_{C'}}(b)
 = 2k-p$. It implies that $\overline{d^+_{A}}(a) = |A|-1$ and
 $\overline{d^-_{B}}(b)=|B|-1$ and that $a$ anti-dominates $A$ and $B$
 anti-dominates $b$. We deduce that for every $a' \in A$ and $b' \in
 B$ there exists an anti-path from $b'$ to $a'$ and applying the same
 arguments to $a'$ and $b'$ we can assume that $A$ and $B$ are both
 independent sets of $D^{M_u}$. As $C$ has then to alternate between
 $A$ and $B$, we have $|A|=|B|=p/2$, and in particular $p$ is
 even.\\ Now, let $B'$ be the set of vertices in $V(C')$ which have an
 anti-out-neighbor in $B$. More formally, we define \mbox{$B'=\{c'\in
   V(C') : \exists b \in B$ such that $c'b$ is an anti-arc$\}$}. For
 any $b \in B$ we have $\overline{d^-_{C'}}(b)=k-p/2$ implying that
 $|B'| \geq k-p/2$.  Similarly, we define $A'=\{c'\in C' : \exists a
 \in A, ac'$ is an anti-arc$\}$. We have the analogous result $|A'|
 \geq k-p/2$.\\ Assume first that there is an arc $b'z'$ of $C'$ with
 $b'\in B'$ and $z'$ having an anti-in-neighbor $z$ in $C$. Then $b'$
 has an anti-out-neighbor $b$ in $C$ and there exists an anti-path $P$
 from $b$ to $z$ in $D^{M_u}[C]$. Moreover, as $D^{M_u}[C']$ is
 strongly connected, there exists also an anti-path $Q$ from $z'$ to
 $b'$ in $D^{M_u}[C]$ and we can conclude with
 Claim~\ref{claim:good-anti-cycle} applied on the anti-cycle $P\cup
 Q$.\\ Now we can assume that every arc $b'z'$ of $C'$ with $b'\in B'$
 satisfies $\overline{N^-}(z')\cap V(C)=\emptyset$.  Symmetrically, we
 can assume that every arc $z'a'$ of $C'$ with $a'\in A'$ satisfies
 $\overline{N^+}(z')\cap V(C)=\emptyset$. But now, we will obtain the
 contradiction that $(C_1,C_2)$ cannot satisfy property ${\cal P}_{\rm
   up}$. Indeed, in particular, we have $A'\cap C'(B') =\emptyset$ and
 as $|A'|\ge k-p/2$ and $|C'(B')|=|B'|\ge k-p/2$ we obtain that
 $(A',C'(B'))$ is a partition of $V(C')$. As $(C_1,C_2)$ satisfies
 property ${\cal P}_{\rm up}$, the cycle $C'$ should contains two
 vertices $u$ and $v$ such that $v$ is the successor of $u$ along
 $C'$, $u$ has an anti-out-neighbor in $C$ and $v$ has an
 anti-in-neighbor in $C$. By the previous arguments, we must have
 $u\notin B'$ and $v\notin A'$, a contradiction to the 
fact that $(A',C'(B'))$ is a partition of $V(C')$.
 \smallskip

\noindent
{\bf Case 2.2 :} {\it For every $a \in A$ and $b\in B$, there are no
  anti-paths from $b$ to $a$}.  Thus, the set $A$ dominates the set
$B$ and $B$ dominates $A$. Suppose without loss of generality that
$|B| \leq |A|$ and let $b$ be a vertex in $B$. If $
\overline{d^+_{C'}}(b) + \overline{d^-_{C'}}(b) > 2k-p$, then we can
find two vertices $u$ and $v$ in $C'$ such that $ub$ and $bv$ are
anti-arcs, and $v$ is the successor of $u$ along $C'$. In that case we
conclude with Claim~\ref{claim:good-anti-cycle} considering the
anti-cycle $P\cup b$ where $P$ is an anti-path from $v$ to $u$ in
$D^{M_u}[C']$ (which exists as $\overline{D^{M_u}[C']}$ is strongly
connected). Therefore, we can assume that for every $b\in B,
\overline{d^+_{C'}}(b) + \overline{d^-_{C'}}(b) \leq 2k-p$. Thus, we
have
 \begin{equation*}
  \begin{split}
   \overline{d^+_{B}}(b) + \overline{d^-_{B}}(b) & =  (\overline{d^+}(b) + \overline{d^-}(b)) - (\overline{d^+_{C'}}(b) + \overline{d^-_{C'}}(b))\\
   & \geq (2k-2)-(2k-p)\\
   & \geq p-2
  \end{split}
 \end{equation*}
 Finally, since $\overline{d^+_{B}}(b) + \overline{d^-_{B}}(b) \leq
 2|B| - 2 \leq |A|+|B|-2 = p - 2$, we have equalities everywhere in
 the previous computation. In particular, we have $|A|=|B|=p/2$ and
 $p$ is even. Moreover, for every $b\in B$ we have
 $\overline{d^+_{B}}(b) = \overline{d^-_{B}}(b)= p/2 - 1$, implying
 that $B$ is an independent set. Symmetrically for $A$, as $|A|=|B|$
 either we can conclude as previously with
 Claim~\ref{claim:good-anti-cycle} or $A$ is also an independent
 set. In this latter case, $C$ would induce a complete bipartite graph
 in $D^{M_u}$ and $D[C_1]$ would be isomorphic to $F_{2p}$, a
 contradiction to our induction hypothesis.\\

To conclude this section, notice that if
$\overline{D^{M_u}[C_2^{M_u}]}$ has exactly one terminal strong
component, then we can conclude similarly. Indeed, we have seen that
if $\overline{D^{M_u}[C_2^{M_u}]}$ is strongly connected, then $D$
admits a good cycle-factor, and if $\overline{D^{M_u}[C_2^{M_u}]}$ is
not strong, then there exists an arc of $C_2^{M_u}$ leaving its unique
terminal component and we can directly conclude with
Claim~\ref{claim:external-arc}.

\subsection{Case B: $D[C_2]$ is isomorphic to $F_{2k-p}$.}

As $D[C_1]$ is not isomorphic to $F_p$, by Claim~\ref{claim:one-cfc}
we can assume without loss of generality that
$\overline{D^{M_u}[C^{M_u}]}$ has exactly one strong initial
component.  As usual, we denote $C_1^{M_u}$ by $C$ and $C_2^{M_u}$ by
$C'$. As $D[C_2]$ is isomorphic to $F_{2k-p}$, $p$ is even and the
digraph $D^{M_u}[C']$ is the complete bipartite digraph
$K_{(k-p/2,k-p/2)}$. We denote by $(A,B)$ its bipartition.

\begin{claim}
 \label{claim:C'F-claim1}
If there exists an arc $ab$ of $C$ such that there is an anti-path in
$C$ from $b$ to $a$, an anti-arc from $a$ to $A$ and an anti-arc from
$A$ to $b$, then $D$ admits a good cycle-factor.
\end{claim}

\begin{proof}
First denote by $a'$ the end of an anti-arc from $a$ to $A$ and by
$b'$ the beginning of an anti-arc from $A$ to $b$. We can assume that
there exists a vertex $c$ of $B$ such that $b'ca'$ is a subpath of
$C'$.  Indeed, as $D^{M_u}[C']$ is isomorphic to a complete bipartite
digraph, there exists a hamiltonian cycle $C''$ of $D^{M_u}[C']$
starting with $b'ca'$. We can then consider $C''$ instead of $C'$.  So
we assume that $b',c$ and $a'$ are consecutive along $C'$, and as $a'$
and $b'$ belong to $A$, $a'b'$ is an anti-arc of $D^{M_u}$.  So we
perform a switch exchange along the anti-cycle formed by the anti-path
from $b$ to $a$ in $C$, and the anti-arcs $aa'$, $a'b'$ and $b'b$. See
Figure~\ref{fig:casa21} which depicts this subcase. According to
Lemma~\ref{lem:switch}, we denote by $M'$ the new contracted matching
in $D$.
\begin{figure}
	\centering
	\includegraphics[width=0.8\linewidth]{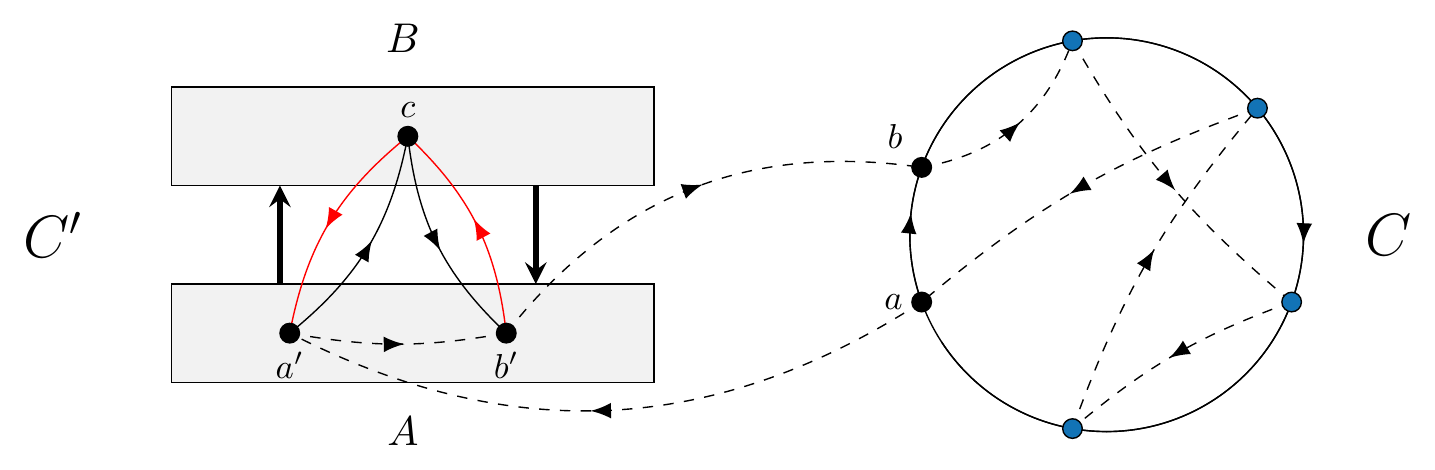}
	\caption{Illustrative example of the Claim~\ref{claim:C'F-claim1}.}
	\label{fig:casb}
\end{figure}
By Lemma~\ref{lem:switch} and Corollary~\ref{cor:switch-cycle-factor},
it is easy to see that we obtain a new cycle-factor $\cal C$ of
$D^{M'}$ containing the 3-cycle $C_3=a'bc$, a cycle $C_s$ containing
all the vertices of $C'\setminus \{a',c\}$ and other cycles included
into $C$. Notice that $D^{M'}[C_s]$ is a bipartite complete digraph on
$2k-p-2$ vertices and we can replace it with two cycles : $C_s'$ on
$p-2$ vertices and $C_s''$ on $2k-2p$ vertices (with $2k-2p\ge 2$).
Now, as $C_3$ contains a vertex of $A$ and $C_s'$ a vertex of $B$, the
union of these two cycles is strongly connected in $D^{M'}$. So using
Theorem~\ref{theo:haggvist-manoussakis} in $D^{M_u}$ there exists a
cycle of length $p+1$ spanning $C_s'\cup C_3$. Using this cycle and
the cycles of $({\cal C}\setminus \{C_3,C_s\})\cup C_s'$, by
Lemma~\ref{lem:cf-to-2cf} we form a 2-cycle-factor of $D^{M'}$ with
one being of length $p+1$. In particular, as $p+1$ is odd, this
cycle-factor of $D^{M'}$ corresponds to a good cycle-factor of $D$.
\end{proof}

As $A$ and $B$ have a symmetric role, Claim~\ref{claim:C'F-claim1}
still holds by replacing $A$ by $B$ in its statement.  Then, let us
consider two cases, according to the structure of $\overline{D^{M_u}[C]}$.
\medskip

\noindent
{\bf Case 1: $\overline{D^{M_u}[C]}$ is strongly connected} In this
case, let $v$ be a vertex of $C$ and $u$ its predecessor along $C$.
As $|C|=p<k$, the vertex $u$ is the beginning of an anti-arc which
ends in $C'$, that is in $A$ or $B$. So, if we cannot conclude with
Claim~\ref{claim:C'F-claim1}, it means than $v$ is dominated by $A$ or
$B$. As this is true for every vertex $v$, by a direct counting
argument, there exist a set $X_A$ of $p/2$ vertices of $C$ which are
dominated by $A$ and a set $X_B$ of $p/2$ vertices which are dominated
by $B$. Moreover, no vertex $x$ of $A$ dominates a vertex of $X_B$
(otherwise we would have $d_{D^{M_u}}(x)>k$) and no vertex of $B$
dominates a vertex of $X_B$. So, for every vertex $v$ of $C$ we have
$d^-_C(v)=k-(k-p/2)=p/2$.\\ Using the same argument between a vertex
$v$ and its successor along $C$ we obtain that every vertex $v$ of $C$
satisfies $d^+_C(v)=p/2$. It means that in $D$, the bipartite
tournament $D[C_1]$ contains $2p$ vertices, is not isomorphic to
$F_{2p}$ and satisfies $d^+(u)=d^-(u)=p$ for each of its vertex
$u$.\par Now, assume first that $p\ge 5$. By induction, provided that
$p\ge 5$, the bipartite tournament $D[C_1]$ has at least 10 vertices
and admits a 2-cycle-factor $(C_{\rm ind},C_{\rm ind}')$ with $C_{\rm
  ind}$ being of length 6.  In $D^{M_u}$, we have a 2-cycle-factor
$(F_{\rm ind}, F_{\rm ind}')$ with $F_{\rm ind}$ being of length
3. Let $xy$ be an arc of $F_{\rm ind}$. As $d^+_{C'}(x)\ge k-p>0$ and
$d^-_{C'}(y)\ge k-p>0$, there exist $x'$ and $y'$ in $C'$ such that
$xx'$ and $y'y$ are arcs of $D^{M_u}$. Now, $D^{M_u}[C']$ being a
complete bipartite digraph on $2k-p$ vertices, it admits a
2-cycle-factor $(C_s,C'_s)$ such that $C_s$ contains $x'$ and $y'$ and
is of length $p-2$. So, using Lemma~\ref{lem:SommeCF}, as
$D^{M_u}[C_s\cup F_{\rm ind}]$ is strongly connected, $D$ admits
a good cycle-factor.\par To conclude, assume that $p=4$. As every
vertex $x$ of $C$ satisfies $d^+_C(x)=d^-_C(x)=2$ and
$\overline{D^{M_u}[C]}$ is strongly connected, it is easy to see that
the anti-arcs of $C$ form a cycle of length 4.  If $C$ contains a
vertex $x$ such that $x$ has an in-neighbor $x'$ in $A$ and an
out-neighbor $y'$ in $B$, then we form a good cycle-factor of
$D^{M_u}$ with : a cycle of length 3 in $C\setminus x$, a cycle of
length $p+1=5$ containing $x'$, $x$, $y'$ and another vertex $x''$ of
$A$ and another vertex $y''$ of $B$, and a cycle of length $2k-8$
covering $C'\setminus \{x',x'',y',y''\}$. Otherwise, it means that we
can write $V(C)=\{a,a',b,b'\}$ such that the in- and out-neighborhood
of $a$ and $a'$ in $C'$ are exactly $A$ and that the in- and
out-neighborhood of $b$ and $b'$ in $C'$ are exactly $B$. Moreover, as
every pair of vertices are linked by at least one arc in $D^{M_u}[C]$,
we can assume that $aa'$ and $bb'$ are arcs of $D^{M_u}$. Then, as
$k>p$, we have $2k-p\ge 6$ and we can select three vertices $a_1, a_2,
a_3$ in $A$ and three vertices $b_1, b_2, b_3$ in $B$. Then, we form a
good cycle-factor by applying Lemma~\ref{lem:cf-to-2cf} to the cycles $aa'a_1b_2a_3$ and
$bb'b_1a_2b_3$ of length $p+1=5$ and a cycle covering $C'\setminus
\{a_1,a_2,a_3,b_1,b_2,b_3\}$.
\medskip

\def\S{Y}

{\bf Case 2: $\overline{D^{M_u}[C]}$ is not strongly connected}.
However, as $(C_1,C_2)$ satisfies ${\cal P}_{\rm up}$, we know that
$\overline{D^{M_u}[C]}$ contains only one strong initial component,
denoted by $\S_1$.  Let $\S_2$ be $V(C)\setminus \S_1$. In particular,
all the arcs from $\S_2$ to $\S_1$ exist in $D^{M_u}$. Moreover, there
exists an arc $xy$ of $C$ with $x\in \S_2$ and $y\in \S_1$. As $\S_1$
is the only initial component of $\overline{D^{M_u}[C]}$ there exists
an anti-path from $y$ to $x$ in $\overline{D^{M_u}[C]}$. As $|C|=p<k$,
there exists a vertex $x'$ in $C'$ such that $xx'$ is an anti-arc of
$D^{M_u}$. Without loss of generality, we can assume that $x'\in
A$. By Claim~\ref{claim:C'F-claim1}, if $y$ has an anti-in-neighbor in
$A$, then $D$ admits a good cycle-factor. So we can assume that $A$
dominates $y$. Similarly, $y$ has an anti-in-neighbor in $C'$ which
must be in $B$ and if we cannot conclude with
Claim~\ref{claim:C'F-claim1}, it means that $x$ dominates $B$. As $x$
dominates also $\S_1$ and $y$ is dominated by $\S_2$, we have
$|\S_1|=|\S_2|=p/2$ and the out-neighborhood of $x$ is exactly $B\cup
S_1$ and the in-neighborhood of $y$ is exactly $A\cup \S_2$. Now,
assume that $z$, the successor of $y$ along $C$ is in $\S_1$.  Notice
that $zy$ is an anti-arc of $D^{M_u}$. As $\S_2$ and $y$ dominate $z$,
it has at least one anti-in-neighbor in $A$ and at least one
anti-in-neighbor in $B$ (otherwise, we would have $d^-(z)\ge
k-p/2+p/2+1=k+1$). As $y$ has an anti-out-neighbor in $C'$, wherever
it is, in $A$ or $B$, we can conclude with
Claim~\ref{claim:C'F-claim1}. Otherwise, it means that $z$ is in
$\S_2$. Symmetrically, the predecessor of $x$ along $C$ is in
$\S_1$. Repeating the argument, we conclude that $C$ alternates
between $\S_1$ and $\S_2$. Indeed $C$ cannot induce a path of positive
length in $\S_1$ for instance : the first vertex of such a path would
be the end of an arc from $\S_2$ to $\S_1$ and the second vertex of
the path would lie in $\S_1$ then. Also the conclusions we had for $x$
and $y$ respectively hold for all vertices of $\S_1$ and $\S_2$. So,
we have $\S_1$ is dominated by $A$ and $\S_2$ and is an independent
set of $D^{M_u}$, and $\S_2$ dominates $B$ and $\S_1$ and is also an
independent set of $D^{M_u}$. We deduce that the in-neighborhood of
$B$ is exactly $A\cup \S_2$ and then the out-neighborhood of $\S_1$ is
$A\cup \S_2$ also. In particular, $\S_1$ dominates $\S_2$ and there is
no anti-path from $\S_1$ to $\S_2$, providing a contradiction, as
$\S_1$ is the only initial strong component of
$\overline{D^{M_u}[C]}$.

 \subsection{Case C: $(C_1,C_2)$  satisfies ${\cal Q}_{\rm down}$}

As we assume that we are not in Case~A, the digraph
$\overline{D^{M_u}[C_2^{M_u}]}$ has at least two initial strong
components, and at least two terminal strong components (as noticed at
the end of Case A).\\
Besides, if $(C_1,C_2)$ satisfies ${\cal P}_{\rm
  down}$, then by exchanging the role of $S$ and $T$, we are in the
symmetrical case of Case A. Then, we can assume that $(C_1,C_2)$ does
not satisfy ${\cal P}_{\rm down}$. Once again, we denote $C_1^{M_u}$
by $C$ and $C_2^{M_u}$ by $C'$. So, by Claim~\ref{claim:not-Pdown},
for every vertex $x$ of $C'$ either there is no arc from $x$ to $C$ or
there is no arc from $C$ to $x$.\\

Before concluding the proof of Theorem~\ref{theo:2cfbis}, we need the
two following claims.


\begin{claim}
\label{claim:noarc-C-z}
If $D^{M_u}[C']$ contains a vertex $x$ such that there is no arc
between $x$ and $C$, then $D$ admits a good cycle-factor.
\end{claim}

\begin{proof}
Otherwise, let $x$ be such a vertex and call $x,y,z,t$ the subpath of
$C'$ of length 3 starting from $x$. We know that either there is no
arc from $C$ to $t$ or no arc from $t$ to $C$. Assume that the latter
holds, the other case could be treated symmetrically.  If there also
exists an anti-arc from $C$ to $t$, then we can find three vertices
$a$, $b$ and $c$ in $C$ such that $a,b,c$ is a subpath of $C$ of
length 2 and that $\{a,x\}$ and $\{c,t\}$ are independent sets of
$D^{M_u}$. So, we will exchange some small paths between $C$ and
$C'$.\par First, notice that $p>3$. Indeed, if we denote by $A_C$
(resp. $B_C$) the set of vertices of $C'$ which anti-dominate $C$
(resp. are anti-dominating by $C$), we have $V(C')=A_C\cup B_C$, $x\in
A_C\cap B_C$ and then $|A_C|+|B_C|\ge |A_C \cup
B_C|+1=|C'|+1=2k-p+1\ge 2k-2$ if $p\le 3$. So, as $|A_C|\le k-1$ and
$|B_C|\le k-1$, we have $|A_C|=|B_C|=k-1$ and $A_C\cap
B_C=\{x\}$. Moreover, we have $t\in A_c\setminus \{x\}$ and so
$\overline{N^+}(c)$ contains $B_C\cup \{t\}$ of size $k$, a
contradiction.\\ Now, to perform the path exchange, let us depict the
situation in $D$: $C_1$ contains the path $aC_1(a)bC_1(b)cC_1(c)$
and $C_2$ contains the path
$xC_2(x)yC_2(y)zC_2(z)tC_2(t)$. Moreover, in $D$, $x$ dominates
$V(C_1)\cap T$, $C_2(x)$ is dominated by $V(C_1)\cap S$, there is an
arc from $c$ to $C_2(t)$ and an arc from $t$ to $C_1(c)$. So, we
replace in $C_1$ the path $aC_1(a)bC_1(b)cC_1(c)$ by
$aC_2(x)yC_2(y)zC_2(z)tC_1(c)$ to obtain the cycle
$\tilde{C_1}$ and we replace in $C_2$ the path $xC_2(x)y$
$C_2(y)zC_2(z)tC_2(t)$ by $xC_1(a)bC_1(b)cC_2(t)$ to obtain
the cycle $\tilde{C_2}$. The cycles $(\tilde{C_1},\tilde{C_2})$ form a
$(2(p+1),4k-2(p+1))$-cycle-factor of $D$. Moreover, $\tilde{C_1}$ is
not isomorphic to $F_{2(p+1)}$. Indeed, as $p>3$, the cycle
$\tilde{C_1}$ contains the predecessor $u$ of $a$ along $C_1$ (which
is not $C_1(c)$ then). Let call $d$ the vertex of $C_1$ with
$C_1(d)=u$. As $C_2(x)$ is dominated by $V(C_1)\cap S$ in $D$, there
is an arc from $d$ to $C_2(x)$. Thus, $\tilde{C_1}$ is not isomorphic
to $F_{2(p+1)}$, as it contains the path $dC_1(d)aC_2(x)$ and the
arc $dC_2(x)$ while $F_{2(p+1)}$ does not contain such a
sub-structure.


So, we can assume that $C$ dominates $t$. Call by $\S$ the strong
component of $\overline{D^{M_u}[C']}$ containing $t$. There exists
$\S_{\rm term}$ a terminal strong component of
$\overline{D^{M_u}[C']}$ distinct from $\S$. Let $\S'$ be the union of
the strong components of $\overline{D^{M_u}[C']}$ different from
$\S_{\rm term}$ and $\S$. As $\S_{\rm term}$ is a terminal strong
component of $\overline{D^{M_u}[C']}$, for any vertex $u$ of $\S_{\rm
  term}$, we have $k-1-|C|\le \overline{d^+_{C'}}(u) \le |\S_{\rm
  term}|-1$ and then $|\S_{\rm term}|\ge k-p$ (noticed that a
symmetrical reasoning holds for initial strong components
also). Moreover, as $C$ and $\S_{\rm term}$ dominate $t$ we must have
exactly $|\S_{\rm term}|= k-p$ and the in-neighborhood of $t$ is
exactly $C\cup \S_{\rm term}$. In particular, $z$ belongs to $\S_{\rm
  term}$ and $\S$ is the other terminal strong component of
$\overline{D^{M_u}[C']}$. Moreover, as $\S'\cup \S$ has size $k$ and
is dominated by $\S_{\rm term}$, then $\S'\cup \S$ is exactly the
out-neighborhood of each vertex of $\S_{\rm term}$. Thus there is no
arc from $\S_{\rm term}$ to $C$. We look at two cases to conclude the
proof.

First assume that $C$ dominates $z$. As $z$ is dominated by $\S$
(recall that $\S$ is a terminal strong component of
$\overline{D^{M_u}[C']}$) and $\S$ has size at least $k-p$, then we
have that $\S\cup C$ is exactly the in-neighborhood of $z$ and that
$|\S|=k-p$. Finally, $\S'$ is non empty (of size $2k-2p$) and there is
no arc from $\S'$ to $\{z,t\}$. To conclude, let $uv$ be an arc of
$C'$ with $u\in \S_{\rm term}\cup \S$ and $v\in \S'$ (such an arc
exits as $\S'\neq \emptyset$). By the previous arguments there exists
an anti-path from $v$ to $u$ in $\overline{D^{M_u}[C']}$ and by
Claim~\ref{claim:external-arc}, we conclude that $D$ admits a good
cycle-factor.

Now, if there is an anti-arc from $C$ to $z$, then we will conclude
as previously using a small path exchange between $C$ and
$C'$. Indeed, if such an anti-arc exists between a vertex $b'$ of $C$
and $z$, call $a'$ the predecessor of $b'$ along $C$. So, $\{a',x\}$
and $\{b',z\}$ are independent sets of $D^{M_u}$. Then, in $D$ we
exchange the path $a'C(a')b'C(b')$ of $C_1$ with the path
$a'C'(x)yC'(y)zC'(z)$ to obtain the cycle
$\tilde{C_1}$. Similarly, we exchange the path
$xC'(x)yC'(y)zC'(z)$ of $C_2$ with the path $xC(a')b'C'(z)$ to
obtain the cycle $\tilde{C_2}$. Then, $(\tilde{C_1},\tilde{C_2})$
forms a good cycle-factor of $D$, as in particular, denoting by $c'$
the predecessor of $a'$ along $C$, $\tilde{C_1}$ contains in $D$ the
path $c'C(c')a'C'(x)$ and the arc $c'C'(x)$ and so $\tilde{C_1}$ is
not isomorphic to $F_{2(p+1)}$ which does not contain such a
sub-structure.

\end{proof}

The last claim will show that every arc of $C'$ is contained in a
digon.

\begin{claim}
\label{claim:C'digoned}
If $D^{M_u}[C']$ contains an arc $xy$ such $yx$ is not an arc of
$D^{M_u}$, then $D$ admits a good cycle-factor.
\end{claim}

\begin{proof}
Assume that $xy$ is an arc of $C'$ such that $yx$ is an anti-arc of
$D^{M_u}$. If $x$ and $y$ are not in the same strong component of
$\overline{D^{M_u}[C']}$, then we conclude with
Claim~\ref{claim:external-arc}.  Otherwise, we assume that $x$ and $y$
lie in a same strong component $\S$ of $\overline{D^{M_u}[C']}$.  As
$\overline{D^{M_u}[C']}$ has at least two initial and two terminal
strong components, there exist an initial strong component $\S_{\rm
  init}$ and a terminal strong component $\S_{\rm term}$ of
$\overline{D^{M_u}[C']}$ which are different from $\S$ (with possibly
$\S_{\rm init}=\S_{\rm term}$).  As previously noticed, we have
$|\S_{\rm init}|\ge k-p$ and $|\S_{\rm term}|\ge k-p$. And as $\S_{\rm
  init}$ and $\S_{\rm term}$ are different from $\S$ then, $\S$
dominates $\S_{\rm init}$ and is dominated by $\S_{\rm term}$.  In
particular, as $x$ and $y$ lie in $\S$ and $xy$ is an arc of $C'$, $x$
has at least an anti-out-neighbor $x'$ in $C$ and $y$ has at least an
anti-in-neighbor $y'$ in $C$ (otherwise, considering $\S_{\rm init}$
or $\S_{\rm term}$ we would have $d^+_{D^{M_u}}(x)\ge k+1$ or
$d^-_{D^{M_u}}(y)\ge k+1$).  If $x'y$ or $xy'$ is not an arc of
$D^{M_u}$ then we conclude with Claim~\ref{claim:good-anti-cycle},
using an anti-path from $y$ to $x$ in $\S$. So we assume that $x'y$
and $xy'$ are arcs of $D^{M_u}$.  Similarly, if we have
$\overline{d^+_C(x)}+\overline{d^-_C(y)}>p$ then there exists a vertex
$u$ in $C$ such that $xu$ and $uy$ are anti-arcs of $D^{M_u}$ and we
conclude with Claim~\ref{claim:good-anti-cycle}. Thus we assume that
we have $\overline{d^+_C(x)}+\overline{d^-_C(y)}\le p$ and then that
$\overline{d^+_{C'}(x)}+\overline{d^-_{C'}(y)}\ge 2k-p$. If there is
no vertex $z$ in $V(C')$ such that $yxz$ is an anti-cycle, then we can
partition $V(C')\setminus \{x,y\}$ into two sets $X$ and $Y$ such that
$x$ anti-dominates $X$ and $Y$ anti-dominates $y$. Call ${\mathcal
  \S}_X$ the set of strongly connected components of
$\overline{D^{M_u}[C']}$ which are included into $X$.  Notice that
${\mathcal \S}_X$ is not empty as it contains all the terminal strong
component of $\overline{D^{M_u}[C']}$, but does not contain any
initial strong component of $\overline{D^{M_u}[C']}$ . Now, consider
any arc $uv$ of $C'$ going from a component $\S_1$ of ${\mathcal
  \S}_X$ to a component $\S_2$ not belonging to ${\cal \S}_X$. As
$\S_2$ contains a vertex of $\{x,y\}\cup Y$, there exists an anti-path
from $v$ to $x$. And as $\S_1$ belongs to ${\cal \S}_X$, there is an
anti-arc from $x$ to $u$. So using Claim~\ref{claim:external-arc}, we
can conclude that $D$ admits a good cycle-factor.\par Thus we assume
that there exists a vertex $z$ in $V(C')$ such that $yxz$ is an
anti-cycle of $D^{M_u}$. By Claim~\ref{claim:noarc-C-z}, we can assume
that there exists an arc between $z$ and $C$.  Without loss of
generality assume that there is an arc $zz'$ from $z$ to $C$.  We will
perform a switch exchange along the anti-cycle $yxz$ and show that the
2-cycle-factor that we obtain will satisfy ${\cal P}_{\rm down}$ and
${\cal Q}_{\rm down}$. Denote by $M'_u$ the perfect matching of $D$
obtain from $M_u$ by switch exchange along $yxz$ (that is
$M'_u=(M_u\setminus \{xM_u(x),yM_u(y),zM_u(z)\})\cup
\{xM_u(y),yM_u(z),zM_u(x)\}$). So, when performing the switch exchange
along $yxz$, by Lemma~\ref{lem:switch}, we obtain
$N^+_{D^{M'_u}}(x)=N^+_{D^{M_u}}(y)$,
$N^+_{D^{M'_u}}(y)=N^+_{D^{M_u}}(z)$ and
$N^+_{D^{M'_u}}(z)=N^+_{D^{M_u}}(x)$. In $D^{M_u}$, call by $P_1$ the
sub-path of $C'$ going from the successor of $y$ (along $C'$) to the
predecessor of $z$ and by $P_2$ the sub-path of $C'$ going from the
successor of $z$ to the predecessor of $x$.  Then, after the switch
exchange, the cycle $C'$ becomes in $D^{M'_u}$ the cycle
$C''=xP_1zyP_2$. We denote by $C_2'$ its corresponding cycle in
$D$. Notice that the strong components of $\overline{D^{M'_u}[C'']}$
are the same than the ones of $\overline{D^{M_u}[C']}$.  Indeed, as
the anti-cycle $yxz$ becomes the anti-cycle $xyz$ in $D^{M'_u}$, the
permutation of the anti-out-neighborhoods of $x$, $y$ and $z$ does not
affect the strong components of $\overline{D^{M_u}[C']}$ and their
relationships. So, $(C_1,C_2')$ still satisfies ${\cal Q}_{\rm
  down}$. Finally, remind that in $D^{M_u}$, the vertex $z$ has an
out-neighbor $z'$ in $C$ and $y$ an in-neighbor $y'$ in $C$. As we
have $N^+_{D^{M'_u}}(y)=N^+_{D^{M_u}}(z)$, the arcs $y'y$ and $yz'$
belong to $D^{M_u}$. Thus, $(C_1,C_2')$ now satisfies ${\cal P}_{\rm
  down}$ and we can conclude with the symmetrical case of Case~A.
\end{proof}

Finally we can assume that every arc of $C'$ is in a digon. We write
$C'=u_1,\dots ,u_l$ with $l=2k-p$. The indices of vertices of $C'$
will be given modulo $l$.  Then we consider two cases:

\paragraph*{Case 1: $p$ is odd.} Then $l=2k-p$ is also odd. As
$\overline{D^{M_u}[C']}$ is not strongly connected, there exists $i\in
\{1,\dots ,l\}$ such that $u_iu_{i-2}$ is an arc of $D^{M_u}$. Without
loss of generality, we can assume that $u_lu_{l-2}$ is an arc of
$D^{M_u}$. We consider then the set $X=\{u_1,u_5,u_7,\dots
,u_{l-p+1}\}$, that is all the vertices $u_i$ with odd $i$ between 1
and $l-p+1$, except $u_3$. Notice that $X$ has size
$(l-p+1-1)/2-1=k-p$. If there is no arc between $X$ and $u_3$, as
$u_3$ has no arc to $C$ or no arc from $C$, we would have
$\overline{d^+_{D^{M_u}}}(u_3)\ge k$ or
$\overline{d^-_{D^{M_u}}}(u_3)\ge k$, a contradiction. So, there
exists an arc between $u_3$ and some $u_i\in X$. If $i=1$, then we
consider the cycle-factor ${\cal C'}$ on $2k-p-1$ vertices containing
$C$ and the cycles with vertex sets $\{u_1,u_2,u_3\},\{u_4,u_5\},\dots
,\{u_{l-p-2},u_{l-p-1}\}$ and the cycle-factor ${\cal C}$ on $p+1$
vertices containing the cycles with vertex set
$\{u_{l-p},u_{l-p+1}\},\{u_{l-p+2},u_{l-p+3}\},\dots
,\{u_{l-1},u_{l}\}$. Notice that $D^{M_u}[u_{l-p},\dots ,u_l]$ is
strongly connected and is not a complete bipartite graph, as it
contains the cycle $u_{l-2},u_{l-1}u_l$. So, by
Lemma~\ref{lem:SommeCF}, the digraph $D$ admits a good cycle-factor.
Now, if $i=u_{l-p+1}$, then we consider the cycle-factor ${\cal C}$ on
$p+1$ vertices containing the cycles with vertex sets
$\{u_{l-p+2},u_{l-p+3}\},\dots ,\{u_{l-1},u_{l}\}$, $\{u_1,u_2\}$, and
the cycle-factor ${\cal C'}$ on $2k-p-1$ vertices containing $C$ and
only the cycle with vertex set $\{u_3,u_4,\dots
,u_{l-p},u_{l-p+1}\}$. We conclude as previously.  Finally, if $i\in
\{5,7,\dots ,l-p-1\}$, then we choose ${\cal C}$ to be the
cycle-factor on $p+1$ vertices containing the cycles with vertex set
$\{u_{l-p},u_{l-p+1}\},\{u_{l-p+2},u_{l-p+3}\},\dots
,\{u_{l-1},u_{l}\}$ and ${\cal C'}$ the one containing the cycles with
vertex set $\{u_{1},u_{2}\},\{u_{3},u_4, \dots ,u_{i}\}, \{u_{i+1},u_{i+2}\}\dots
,\{u_{l-p-2},u_{l-p-1}\}$. Once again, we conclude as previously.

\paragraph*{Case 2: $p$ is even.} If there exists an arc between two
vertices at distance 2 along $C'$, then we will proceed almost as in
the case where $p$ is odd. The difference here, is that ${\cal C}$
will contain cycles all of length 2 except one of length 3. Indeed,
assume for instance that $u_lu_{l-2}$ is an arc of $D^{M_u}$. We
consider once again the set $X=\{u_1,u_5,u_7,\dots ,u_{l-p+1}\}$ and
show, as previously, that there exists an arc between a vertex $u_i$
of $X$ and $u_3$. If $i=1$, then we consider the cycle-factor ${\cal
  C'}$ on $2k-p-1$ vertices containing $C$ and the cycles with vertex
sets $\{u_1,u_2,u_3\},\{u_4,u_5\},\dots ,\{u_{l-p-2},u_{l-p-1}\}$ and
the cycle-factor ${\cal C}$ on $p+1$ vertices containing the cycles
with vertex set $\{u_{l-p},u_{l-p+1}\},\{u_{l-p+2},u_{l-p+3}\},\dots
,\{u_{l-4},u_{l-3}\},\{u_{l-2},u_{l-1},u_l\}$. Once again, we conclude
with Lemma~\ref{lem:SommeCF} that $D$ admits a good cycle-factor. The
cases where $i=l-p+1$ and where $i\in \{5,7,\dots ,l-p-1\}$ are
similar to the corresponding cases where $p$ is odd.\par Finally,
assume that there is no arc between two vertices at distance 2 along
$C'$. Then, we denote by $A$ the vertices $u_i$ with odd indices and
by $B$ the vertices $u_i$ with even indices. The sets $A$ and $B$ form
two strong connected components of $\overline{D^{M_u}[C']}$ and so
their are both initial and terminal strong components of
$\overline{D^{M_u}[C']}$. In particular, $D^{M_u}$ contains all the
arcs from $A$ to $B$ and all the arcs from $B$ to $A$. The vertex
$u_1$ must have a neighbor in $A$. Indeed, otherwise, as there is no
arc from $u_1$ to $C$ or from $C$ to $u_1$, we will conclude that
$\overline{d^+_{D^{M_u}}}(u_1)\ge k$ or
$\overline{d^-_{D^{M_u}}}(u_1)\ge k$, which is not possible. So,
assume that $u_1$ is adjacent to a vertex $u_i$ of $A$. Then, instead
$C'$ we consider the cycle $C''=u_1,u_2,u_i,u_4,\dots
,u_{i-1}u_3u_{i+1}, \dots u_l$. As there exist all the possible arcs
between $A$ and $B$, then all the arcs of $C''$ are contained in a
digon and there exists an arc between two vertices at distance 2 along
$C''$ (the arc $u_1u_i$). Thus, we are in the previous case and we
conclude that $D^{M_u}$ admits a good anti-cycle.

\section{Concluding remarks}
\label{sec:cr}
We finish this paper with some conjectures about the problem of
cycle-factor in bipartite or multipartite tournaments. First of all,
we have to mention the two related conjectures appearing in the
original paper of Zhang and \emph{al.}~\cite{ZMS94}. The first one
adds a new hypothesis imposing an arc in the 2-cycle-factor.
\begin{conjecture}[Zhang and al.~\cite{ZMS94}]
Let $D$ be a $k$-regular bipartite tournament, with $k$ an integer
greater than 2. Let $uv$ be any specified arc of $D$. If $D$ is
isomorphic neither to $F_{4k}$ nor to some other specified families of
digraphs, then for every even $p$ with $4\leq p \leq |V(D)| -4$, $D$
has a cycle $C$ of length $p$ such that $D\setminus C$ is hamiltonian
and such that $C$ goes through the arc $uv$.
\end{conjecture}

The second conjecture, conversely, imposes that the cycles contain
specific vertices.
\begin{conjecture}[Zhang and al.~\cite{ZMS94}]
Let $D$ be a $k$-regular bipartite tournament, with $k$ an integer
greater than 2. Let $u$ and $v$ be two specified vertices of $D$. If
$D$ is isomorphic neither to $F_{4k}$ nor to some other specified
families of digraphs, then for every even $p$ with $4\leq p \leq
|V(D)| -4$, $D$ has a cycle $C$ of length $p$ such that $D\setminus C$
is hamiltonian and such that $C$ contains $u$ and $D\setminus C$
contains $v$.
\end{conjecture}

Throughout the proof of Theorem~\ref{theo:2cf}, we intensively used
the regularity of the bipartite tournament. It seems that we cannot
get rid of this condition as we can easily find an infinite family of
bipartite tournament with $|d^+(u) - d^-(u)| \leq 1$ for every vertex
$u$ and $|d^+(u) - d^+(v)| \leq 1$ for every pair of vertices $\{u,
v\}$, which does not contain any cycle-factor. For instance, for any
$k\ge 1$ consider the bipartite tournament, inspired by $F_{4k}$,
consisting of four independent sets $K$, $L$, $M$ and $N$ with
$|K|=|N|=k$ and $|L|=|M|=k+1$ with all possible arcs from $K$ to $L$,
from $L$ to $M$, from $M$ to $N$ and from $N$ to $K$.\\


Let $D$ be a $c$-partite tournament, and denote by $I_1, \dots, I_c$
its independent sets. We say that $D$ is $k$-\emph{fully regular} if,
for any distinct $i$ and $j$ with $1\leq i,j \leq c$, $D[I_i \cup
  I_j]$ is a $k$-regular bipartite tournament. In particular, all the
sets $I_i$ have size $2k$.\\ In~\cite{Yeo99}, Yeo proved that if $c\ge
5$ then, in every $c$-partite regular tournament $D$, every vertex is
contained in a cycle of length $l$ for $l=3,\dots ,|V(D)|$. He also
conjectured the following.

\begin{conjecture}[Yeo~\cite{Yeo99}]
 Every regular $c$-partite tournaments $D$, with $c\ge 5$, contains a
 $(p,|V(D)|-p)$-cycle-factor for all $p\in \{3,\dots ,|V(D)|-3\}$.
\end{conjecture}

An extension of our results and a weaker form of Yeo's Conjecture could
be the following.

\begin{conjecture}
 \label{conj:mult}
Let $D$ be a $k$-fully regular $c$-partite tournament with $c\ge
5$. Then for every even $p$ with $4\leq p \leq |V(D)| -4$, $D$ has a
$(p, |V(D)| - p)$-cycle-factor.
\end{conjecture}

We can see that if $c$ is even and there is at least one pair $\{I_i,
I_j\}$ such that $D[I_i, I_j]$ is not isomorphic to $F_{4k}$, then our
result implies the Conjecture~\ref{conj:mult}, by properly
partitioning the sets $I_i$ into two parts and applying
Theorem~\ref{theo:2cf} on the bipartite lying between the two
parts. However, the case where $c$ is odd seems more complicated to handle.
\vspace{15pt}

\bibliography{RefBip}

\begin{thebibliography}{10}

\bibitem{BLH14}
Yandong Bai, Hao Li, and Weihua He.
\newblock Complementary cycles in regular bipartite tournaments.
\newblock {\em Discret. Math.}, 333:14--27, 2014.

\bibitem{BJG09}
J{\o}rgen Bang{-}Jensen and Gregory~Z. Gutin.
\newblock {\em Digraphs - Theory, Algorithms and Applications, Second Edition}.
\newblock Springer Monographs in Mathematics. Springer, 2009.

\bibitem{BT81}
Jean{-}Claude Bermond and Carsten Thomassen.
\newblock Cycles in digraphs- a survey.
\newblock {\em J. Graph Theory}, 5(1):1--43, 1981.

\bibitem{BM08}
J.~Adrian Bondy and Uppaluri S.~R. Murty.
\newblock {\em Graph Theory}.
\newblock Graduate Texts in Mathematics. Springer, 2008.

\bibitem{C59}
Paul Camion.
\newblock {Chemins et circuits hamiltoniens des graphes complets}.
\newblock {\em C. R. Acad. Sci. Paris}, 249(21):2151--2152, 1959.

\bibitem{C01}
Guantao Chen, Ronald~J. Gould, and Hao Li.
\newblock Partitioning vertices of a tournament into independent cycles.
\newblock {\em J. Comb. Theory, Ser. {B}}, 83(2):213--220, 2001.

\bibitem{HM89}
Roland H{\"{a}}ggkvist and Yannis Manoussakis.
\newblock Cycles and paths in bipartite tournaments with spanning
  configurations.
\newblock {\em Comb.}, 9(1):33--38, 1989.

\bibitem{KO16}
Daniela K{\"{u}}hn, Deryk Osthus, and Timothy Townsend.
\newblock Proof of a tournament partition conjecture and an application to
  1-factors with prescribed cycle lengths.
\newblock {\em Comb.}, 36(4):451--469, 2016.

\bibitem{LS05}
Hao Li and Jinlong Shu.
\newblock The partition of a strong tournament.
\newblock {\em Discret. Math.}, 290(2/3):211--220, 2005.

\bibitem{M87}
Yannis Manoussakis.
\newblock {\em Problèmes extrémaux dans les graphes orientés}.
\newblock PhD thesis, 1987.
\newblock 1987PA112305.

\bibitem{M66}
J.~W. Moon.
\newblock On subtournaments of a tournament.
\newblock {\em Can. Math. Bull.}, 9(3):297–301, 1966.

\bibitem{R85}
K.~Brooks Reid.
\newblock {Two complementary circuits in two-connected tournaments}.
\newblock In {\em {Cycles in graphs (Burnaby, B.C., 1982)}}, volume 115 of {\em
  North-Holland Math. Stud.}, pages 321--334. North-Holland, 1985.

\bibitem{S93}
Zeng~Min Song.
\newblock Complementary cycles of all lengths in tournaments.
\newblock {\em J. Comb. Theory, Ser. {B}}, 57(1):18--25, 1993.

\bibitem{Yeo99}
Anders Yeo.
\newblock Diregular \emph{c}-partite tournaments are vertex-pancyclic when
  \emph{c} {\(\geq\)} 5.
\newblock {\em Journal of Graph Theory}, 32(2):137--152, 1999.

\bibitem{ZMS94}
Ke~Min Zhang, Yannis Manoussakis, and Zeng~Min Song.
\newblock Complementary cycles containing a fixed arc in diregular bipartite
  tournaments.
\newblock {\em Discret. Math.}, 133(1-3):325--328, 1994.

\bibitem{SZ88}
Ke~Min Zhang and Zeng~Min Song.
\newblock {Complementary cycles containing of fixed vertices in bipartite
  tournaments}.
\newblock {\em Appl. J. Chin. Univ.}, 3:401--407, 1988.

\end{thebibliography}

\end{document}